\documentclass[11pt, reqno]{amsart}

\usepackage[margin=3cm]{geometry}

\usepackage[T1]{fontenc}
\usepackage[utf8]{inputenc}
\usepackage[usenames,dvipsnames]{color}

\usepackage[american]{babel}

\usepackage[citestyle=alphabetic,bibstyle=alphabetic]{biblatex}

\usepackage{amsmath}
\usepackage{amsthm}
\usepackage{mathabx}	
\usepackage{amscd}
\usepackage{color}
\usepackage[all,cmtip]{xy}
\usepackage{hyperref}
\usepackage{tikz-cd}
\usepackage{enumitem}
\usepackage{dynkin-diagrams}
\usepackage{cleveref}
\usepackage{colortbl}

\newtheoremstyle{thms}{0.2em}{0.2em}{\itshape}{}{\bfseries}{.}{ }{}

\theoremstyle{plain}
\theoremstyle{definition}

\newtheorem{theorem}{Theorem}[section]
\newtheorem{lemma}[theorem]{Lemma}
\newtheorem{claim}[theorem]{Claim}
\newtheorem{corollary}[theorem]{Corollary}
\newtheorem{definition}[theorem]{Definition}
\newtheorem{proposition}[theorem]{Proposition}
\newtheorem{remark}[theorem]{Remark}
\newtheorem{conjecture}[theorem]{Conjecture}

\newtheorem{definition-proposition}[theorem]{Definition-Proposition}

\DeclareMathOperator{\Ad}{\mathrm{Ad}}			
\DeclareMathOperator{\Aut}{\mathrm{Aut}}		
\DeclareMathOperator{\Dom}{Dom} 

\DeclareMathOperator{\End}{End} 
	
\DeclareMathOperator{\Gal}{Gal}	
\DeclareMathOperator{\Id}{Id} 

\DeclareMathOperator{\Norm}{Norm}	
\DeclareMathOperator{\Out}{Out}	
\DeclareMathOperator{\PGL}{PGL}		
\DeclareMathOperator{\pr}{pr}	
\DeclareMathOperator{\Pic}{Pic}    
\DeclareMathOperator{\rk}{rk}			

\DeclareMathOperator{\Res}{Res} 

\DeclareMathOperator{\SL}{SL}			

\DeclareMathOperator{\Sch}{Sch}  
\DeclareMathOperator{\Supp}{Supp} 
\DeclareMathOperator{\Spec}{Spec}		

\title{An equivariant compactification for adjoint reductive group schemes}
\author{Shang Li}
\address{Tsinghua University, Yau Mathematical Sciences Center, Beijing, China}
\email{shangli@tsinghua.edu.cn}
\date{\today}

\begin{document}

\setcounter{tocdepth}{1}

\maketitle

\begin{abstract}
    Wonderful compactifications of adjoint reductive groups over an algebraically closed field play an important role in algebraic geometry and representation theory. In this paper, we construct an equivariant compactification for adjoint reductive groups over arbitrary base schemes. The geometric fibers of our compactifications are the classical wonderful compactifications of De Concini and Procesi. Our construction is based on a variant of the Artin--Weil method of birational group laws, and, in the split case, dose not depend on the existence of the classical wonderful compactification over an algebraically closed field. In particular, our construction gives a new intrinsic construction of wonderful compactifications. The Picard group scheme of our compactifications is computed. We also discuss several applications of our compactifications in the study of torsors under reductive group schemes.
\end{abstract}

\tableofcontents

\pagestyle{plain}

\section{Introduction}

Reductive groups play a fundamental role both in algebraic geometry and in number theory. Seeking for an equivariant compactification of reductive groups is beyond any doubt an important problem, see the ICM report \cite{SpringerICM} for recent developments on this topic. The wonderful compactification for adjoint semisimple groups over algebraically closed fields is a good solution to this problem. The wonderful compactifications have a wide range of applications, for instance, in the study of character sheaves (see, \emph{loc.cit}, \cite{Ginsburg} and \cite{Lusztigparabolicsheaves}) and in the study of local models of Shimura varieties (see \cite{XuhuaHelocalmodel} and \cite[Section~8]{surveylocalmodel}).

In this paper, for an adjoint reductive group scheme $\mathbf{G}$ over a scheme $S$, we construct the following equivariant compactification for $\mathbf{G}:$

\begin{theorem}\label{mainresult}
(\Cref{solutiontoconjecadjoint}). There is a smooth projective $S$-scheme $\mathcal{X}$ containing $\mathbf{G}$ as an open schematically dense subscheme such that $\mathcal{X}$ is equipped with a $(\mathbf{G}\times_S\mathbf{G})$-action which extends the $(\mathbf{G}\times_S \mathbf{G})$-action $((g_1,g_2),g)\mapsto g_1gg_2^{-1}$ on $\mathbf{G}$. Each geometric fiber of $\mathcal{X}$ is identified with the wonderful compactification of the corresponding geometric fiber of $\mathbf{G}$. The boundary $\mathcal{X}\backslash \mathbf{G}$ is a smooth relative Cartier divisor with $S$-relative normal crossings.
\end{theorem}

In particular, if the base scheme $S$ in \Cref{mainresult} is the spectrum of an algebraically closed field, the scheme $\mathcal{X}$ recovers the classical wonderful compactification of $\mathbf{G}$. 

In algebraic geometry, the search for proper smooth schemes over $\Spec(\mathbb{Z})$ is a deep topic and seems to date back to Grothendieck \cite[page 242]{Mazurarithmeticoncurves}. Actually, as shown by Fontaine in \cite{Fontainepropersmoothscheme}, there are some constraints for a smooth proper scheme over $\Spec(\mathbb{Z})$ to exist. Moreover, many famous non-existence results of smooth proper $\Spec(\mathbb{Z})$-scheme have been established, for instance, \cite{Fontainethereisnoabelianscheme} and \cite{SchroerthereisnoenriquessurfaceoverZ}. In this aspect, a remarkable feature of the existence result \Cref{mainresult} is that, for any (not necessarily split) reductive group scheme over $\Spec(\mathbb{Z})$ of adjoint type, we can associate to it a smooth proper scheme over $\Spec(\mathbb{Z})$ whose geometry is closely related to the reductive group scheme. 

Our construction of $\mathcal{X}$ in \Cref{mainresult} is intrinsic and functorial in the sense that we avoid embedding the group $\mathbf{G}$ into an ambient projective space, and we use a sheaf-theoretic argument so that the formation of $\mathcal{X}$ is compatible with any base change. 

In \cite[Conjecture 6.2.3]{CKproblemtorsors}, \v{C}esnavi\v{c}ius raises a question about how to equivariantly embed a reductive group scheme over a base scheme into a projective scheme in which the group is relatively dense with respect to the base scheme. \Cref{mainresult} gives a positive answer to this question for adjoint reductive group schemes. As a consequence of \Cref{mainresult}, we deduce the following lifting property of torsors:

\begin{corollary}
    (\Cref{compactificationtorsor}). Let $A$ be an isotrivial torsor under an adjoint reductive group $\mathbf{G}$ over an affine semilocal scheme $S$. For each $a\in A(Z)$ with $Z$ a closed subscheme of $S$, there exist a finite étale cover $\widetilde{S}$ of $S$, an $S$-morphism
    $\nu\colon Z\rightarrow\widetilde{S}$ and a section $\widetilde{a}\in A(\widetilde{S})$ whose $\nu$-pullback is $a$.
\end{corollary}

Indeed, by a twist argument, we deduce from \Cref{mainresult} an equivariant compactification containing the torsor $A$ as a $S$-dense open subscheme. Then we use the Bertini theorem applied to each closed fiber of the compactification of $A$ to cut out the desired finite étale cover $\widetilde{S}$.
For a further consequence of \Cref{mainresult} that concerns the techniques of equating reductive (not necessarily adjoint) group schemes, see \Cref{equatinggroup}.

\subsection{Classical wonderful compactifications}\label{introclassical}
Let us briefly review the theory of wonderful compactifications.
Consider an adjoint reductive group $G$ over an algebraically closed field $k$.  The wonderful compactification $X$ of $G$ was introduced by De Concini and Procesi \cite{completesymmetricvarieties} in characteristic zero, and by Strickland \cite{strickland} in arbitrary characteristic, and it satisfies the following properties:
\begin{itemize}
    \item $X$ is a smooth projective variety containing $G$ as an open schematically dense subvariety such that $X$ is equipped with an action of $G\times_k G$ that extends the action $(g_1,g_2)\cdot g= g_1 gg_2^{-1}$ of $G\times_k G$ on $G$;
    \item the boundary $X\backslash G$ is a union of $(G\times_k G)$-stable smooth prime divisors $X_i$ ($i=1,...,l$) with normal crossings;
    \item the closures of the $(G\times_k G)$-orbits in $X\backslash G$ are the partial intersections of the $X_i$.
\end{itemize}
The idea of their constructions is to define $X$ as the closure of $G$ in a projective space constructed via representation theory. In \cite{compactificationHilbertsch}, Brion gave an algebro-geometric realization of $X$ by taking the closure of $G$ in the Hilbert scheme of a certain flag variety of $G\times G$. 

To investigate the arithmetic aspects of the wonderful compactifications, a basic question is to seek an integral model of wonderful compactifications. In \cite{Faltingsresolution}, Faltings sketched an integral model of wonderful compactifications (for more general symmetric varieties) over $\Spec (\mathbb{Z}_p)$, where $p$ is the characteristic of the base field $k$, see also \cite{SpringerConcinicompactification}. In an unpublished work, Gabber constructed the wonderful compactification for a Chevalley group over $\Spec(\mathbb{Z})$.

The main feature of the wonderful compactification $X$ that we are going to make use of is that $X$ contains an open subvariety $\overline{\Omega}$ which is usually called the big cell of $X$ and is isomorphic to the affine space $U^-\times_k \mathbb{A}^l_k\times_k U^+ $, where $l$ is the rank of $G$ and $U^+$ (resp., $U^-$) is the unipotent radical of a Borel subgroup (resp., of an opposite Borel subgroup). The geometry of $X$ is, in many respects, largely controlled by $\overline{\Omega}$ mainly because translates of $\overline{\Omega}$ under the action of $G\times_k G$ cover $X$. For instance, the smoothness of $X$ follows from that of $\overline{\Omega}$ (see, for instance, \cite[Theorem~6.1.8 (i)]{BrionKumar}). Therefore, we may seek to  recover $X$ as a quotient of $G\times_k\overline{\Omega}\times_k G$. One advantage of this is that $\overline{\Omega}$ is much simpler than $X$ in the sense that it has explicit coordinates, and after shrinking if needed, the birational action of $G\times_k G$ on $\overline{\Omega}$ can be defined over $\Spec(\mathbb{Z})$. This suggests for us to define a birational action of $ G\times_k  G$ on $\overline{\Omega}$ over more general base schemes, here a birational action is an analogue of the notion of a birational group law that was first due to Weil in the setting of  varieties, and later was generalized to schemes by Artin (see \cite{groupschemeoutofbirationallaw}).

\subsection{The proof of \Cref{mainresult}}
 The overall strategy of the proof is that we first establish \Cref{mainresult} for split groups and then make a descent argument.

 \subsubsection{Split case}
 Now assume that $\mathbf{G}$ splits over $S$ with respect to a split maximal torus $\mathbf{T}$, choose a Borel subgroup $\mathbf{B}$ containing $\mathbf{T}$, and denote the corresponding based root datum of $\mathbf{G}$ by $(X,\Psi, \Delta, X^{\vee}, \Psi^{\vee}, \Delta^{\vee})$. Then we have two big cells, and an immersion between them modelled on classical wonderful compactifications:
$$\Omega_{\mathbf{G}}\coloneq\mathbf{U}^-\times_S \mathbf{T}\times_S \mathbf{U}^+\longrightarrow \overline{\Omega}\coloneq\mathbf{U}^-\times_S \overline{\mathbf{T}}\times_S \mathbf{U}^+,$$
that is the identity on $\mathbf{U}^+$ and $\mathbf{U}^-$, and sends $\mathbf{T}$ to $\overline{\mathbf{T}}$ via negative simple roots, where $\overline{\mathbf{T}}$ is the product $\prod_{\Delta} \mathbb{G}_a$.
In \Cref{theoremrationalaction}, we obtain a rational action of $\mathbf{G}\times_S \mathbf{G}$ on $\overline{\Omega}$ which is a rational $S$-morphism  $$\pi\colon \mathbf{G}\times_S\overline{\Omega}\times_S\mathbf{G}\dashrightarrow \overline{\Omega}$$
    such that for any $x\in S$, the geometric fiber $\pi_{\overline{k(x)}}$, on a certain large enough open subscheme, coincides with the rational action of $(\mathbf{G})_{\overline{k(x)}}\times_{\overline{k(x)}}(\mathbf{G})_{\overline{k(x)}}$ on the big cell $\overline{\Omega}_{\overline{k(x)}}$ of the wonderful compactification of ${\mathbf{G}}_{\overline{k(x)}}$ over the algebraically closed field $\overline{k(x)}$. A similar rational group law on $\Omega_{\mathbf{G}}$ is used to prove the Existence Theorem in  \cite[exposé~XXV]{SGA3III}. Comparing with the construction of the rational group law on $\Omega_{\mathbf{G}}$ in \emph{loc. cit.,} which assumes the Existence Theorem over $\mathbb{C}$, our construction of $\pi$ does not rely on the classical wonderful compactification.

Using the rational action $\pi$, we define the following equivalence relation (\Cref{relationonOmega}) on $ \mathbf{G}\times_S\overline{\Omega}\times_S \mathbf{G}$:
     for every $S$-scheme $S'$ and $(g_1, x, g_2)$, $(g_1', x', g_2')\in( \mathbf{G}\times_S\overline{\Omega}\times_S \mathbf{G} )(S')$, we say that $(g_1, x, g_2)$ and $(g_1', x', g_2')$ are equivalent if and only if, there exist an fppf cover $S''\rightarrow S'$ and a section $(a_1, a_2)\in ( \mathbf{G}\times_S \mathbf{G})(S'')$   such that $\pi(a_1g_1,x,a_2g_2)$ and $\pi(a_1g_1',x',a_2g_2')$ are both well defined and equal over $S''$.

One advantage of the above definition is that, although $\pi$ is merely defined over an open subscheme of $\mathbf{G}\times_S\overline{\Omega}\times_S\mathbf{G}$, one can use $(\mathbf{G}\times_S \mathbf{G})$-translations to bring two sections into the definition domain of $\pi$ so that they can be compared via $\pi$. We can then define the compactification $\mathcal{X}$ of $\mathbf{G}$ as an fppf-sheaf on $\Sch/S$ by taking the quotient of $\mathbf{G}\times_S\overline{\Omega}\times_S \mathbf{G}$ with respect to this equivalence relation. 

The proof of the fact that $\mathcal{X}$ is a scheme involves two steps. We first reinterpret $\mathcal{X}$ as a quotient of $\mathbf{G}\times_S\overline{\Omega}\times_S\mathbf{G}$ with respect to an fppf relation (\Cref{algebraicspace}), hence by a theorem due to Artin (\cite[Corollaire (10.4)]{Champsalgebrique}), $\mathcal{X}$ is an algebraic space. Secondly, we endow $\mathcal{X}$ with a group action of $\mathbf{G}\times_S\mathbf{G}$ (\Cref{definitionofgrpaction}), so that the $(\mathbf{G}\times_S\mathbf{G})$-translates of the open subscheme $\overline{\Omega}$ of $\mathcal{X}$ (\Cref{representedopenimmer}) cover $\mathcal{X}$. Thus, the schematic nature of $\mathcal{X}$ follows from a Zariski gluing of scheme-like open neighborhoods of points of $\mathcal{X}$, see \cite[Section~6.6, Theorem~2]{BLR}.

The scheme $\mathcal{X}$ is projective over $S$ (\Cref{projectivity}).
Roughly, the quasi-projectivity of $\mathcal{X}$ comes from the relative ampleness of the boundary divisor $\mathcal{X}\backslash \overline{\Omega}$, see \cite[Section~6.6, Theorem~2 (d)]{BLR}. By a variant of the valuative criterion for properness and the Iwahori decomposition, we deduce the properness of $\mathcal{X}$, see \Cref{properness} for details.

To finish the proof of \Cref{mainresult} in the split case,  by the functoriality of the definition of $\mathcal{X}$ and the fiberwise condition on $\pi$, we deduce that the geometric fibers of $\mathcal{X}$ coincide with the wonderful compactifications of the corresponding geometric fibers of $\mathbf{G}$, see \Cref{groupembedtocompac}.

\subsubsection{Non-split case} Now we consider $\mathbf{G}$ as an adjoint reductive group scheme over $S$. By \cite[exposé~XXII, corollaire~2.3]{SGA3III}, we may assume that there exists an étale cover $U$ of $S$ and a split reductive $S$-group $\mathbf{G}_0$, which is isomorphic to $\mathbf{G}$ over $U$. Applying the split case of \Cref{mainresult} to $\mathbf{G}_0$ to produce a compactification $\mathcal{X}_0$, the rest of the proof is to descend $\widetilde{\mathcal{X}}:=\mathcal{X}_0\times_S U$ to get the sought $\mathcal{X}$ for $\mathbf{G}$. 
Note that $\mathbf{G}$ is descended from $\widetilde{\mathbf{G}}:=\mathbf{G}_0\times_S U$ via a \v{C}ech cocycle $\chi$ which is an element of $\Aut_{\mathbf{G}_0/S}(U\times_S U)$ together with a cocycle condition. Then by \cite[exposé~XXIV, théoreme~1.3]{SGA3III}, $\chi$ is the composition of an inner automorphism $\delta$ and an outer automorphism $\eta$. The inner part $\delta$ naturally induces an automorphism of $\widetilde{\mathcal{X}}$ via the group action on $\widetilde{\mathcal{X}}$, and we show that the outer part $\eta$ also induces an automorphism of $\widetilde{\mathcal{X}}$, see \Cref{outmorphismpreserverelation}. Combining these two automorphisms on $\widetilde{\mathcal{X}}$ together, by the schematic density of $\widetilde{\mathbf{G}}$ in $\widetilde{\mathcal{X}}$, we get a descent datum on $\widetilde{\mathcal{X}}$. To make this descent datum effective, we construct an ample divisor compatible with the descent datum, see \Cref{descentampleCartier}. The group action and the fiberwise condition on $\widetilde{\mathcal{X}}$ also descend, see \Cref{solutiontoconjecadjoint}.

\subsection{Divisors and relative Picard schemes}
In \Cref{boundarydivisor}, we show that the Picard scheme of $\mathcal{X}$ in \Cref{mainresult}, when $\mathbf{G}$ is split, is a constant scheme associated to the free abelian group of rank equal to the rank of $\mathbf{G}$. Roughly, base change reduces us to the case when the base $S$ is $\Spec(\mathbb{Z})$, and a fibral criterion \cite[Corollaire 17.9.5]{EGAIV4} further reduces us to the case when $S$ is an algebraically closed field, which follows from the coresponding result for the classical wonderful compactifications (\cite[Lemma 6.19]{BrionKumar}).

We also show that when $\mathbf{G}$ is split, the boundary $\mathcal{X\backslash \mathbf{G}}$ shares some nice properties of classical wonderful compactifications as we recall in \Cref{introclassical}, see \Cref{boundarydivisors}. We also study the Picard group of $\mathcal{X}$ and its generators in certain nonsplit case, see \Cref{picardgroupnonsplit}.

\subsection{Notation and conventions}
All rings are commutative and unital. For the $n$-fold product $X_1\times X_2\times...\times X_n$ of schemes, for $\{i_1,i_2,...,i_l\}$ which is a strict ascending sequence of $\{1,2,...,n\}$, we denote by $\pr_{i_1,i_2,...,i_l}$ the projection from $X_1\times X_2\times...\times X_n$ onto the product of the factors with indices $i_1,i_2,...,i_l$. 

We use dotted arrows to depict a rational morphism. For an $S$-rational morphism $f$ between two schemes $X$ and $Y$ over a scheme $S$, a test $S$-scheme $S'$ and a section $x\in X(S')$, we say that $x$ is well defined with respect to $f$ if the image of $x$ in $X_{S'}$ lies in the definition domain of $f_{S'}$. When $Y$ is separated over $S$, we denote by $\Dom(f)$ the definition domain of $f$. For the existence and the uniqueness of the definition domain of a rational morphism, see \cite[exposé~XVIII, définition~1.5.1]{SGA3II}.

For an effective Cartier divisor $\mathbf{D}$ of a scheme $X$, we write $\mathcal{L}_{X}(\mathbf{D})$ for the dual sheaf of the ideal sheaf of $\mathbf{D}$. 

A reductive group scheme is assumed to be fiberwise connected. 

For an affine scheme $X$ and an $f\in\Gamma (X,\mathcal{O}_X)$, we denote the closed (resp., open) subscheme of $X$ defined by $f$ as $V_X(f)$ (resp., $D_X(f)$).

Let $X$ be a scheme. In this paper, following \cite[définition~11.10.2]{EGAIV3}, we say that a subscheme $j:Y\hookrightarrow X$ is schematically dense if the structural morphism of sheaves $\mathcal{O}_X\rightarrow j_*(\mathcal{O}_Y)$ is injective.

Let $X$ be a scheme over a scheme $S$, and let $Y\subset X$ be an $S$-subscheme. We say that $Y$ is $S$-dense if, for any $S$-scheme $S'$, the base change $Y_{S'}$ is schematically dense in $X_{S'}$, cf. \cite[définition~11.10.8]{EGAIV3} or \cite[exposé~XVIII, \S~1]{SGA3II}. By \cite[théorème~11.10.9]{EGAIV3}, if $X$ is locally noetherian and $Y$ is flat over $S$, then the $S$-density of $Y$ is equivalent to the ``fiberwise density'' of $Y$, i.e., for any point $s\in S$, the fiber $Y_s$ is schematically dense in the fiber $X_s$.

\subsection{Acknowledgements}
I would like to thank my advisor Kęstutis Česnavičius deeply for his instruction on mathematics and writing and his sustained encouragement. I am especially grateful to Ofer Gabber for sharing his insights that are crucial to this paper. I heartfully thank Michel Brion and Alexis Bouthier for several very helpful conversations and suggestions. I thank Jean-Loup Waldspurger for a very critical remark that significantly improved the manuscript during a seminar talk. I thank Timo Richarz, Stéphanie Cupit-Foutou, Ning Guo, Arnab Kundu, Zhouhang Mao, Sheng Chen and Federico Scavia for helpful correspondence. I also thank the referees for
many useful suggestions. This project has received funding from the European Research Council (ERC) under the European Union’s Horizon 2020 research and innovation programme (grant agreement No. 851146).

\section{Recollection of wonderful compactifications}\label{wonderfulcomp}
In this section, we briefly review the classical theory of wonderful compactifications. We fix an algebraically closed field $k$, and if a fiber product is formed over $k$, we will omit the subscript $k$. 

\subsection{The group setup}
Let $G$ be an adjoint reductive group over $k$, whose rank is $l$, and let $G_{\text{sc}}$ be its simply connected covering. Fix a Borel subgroup $B_{\text{sc}}\subset G_{\text{sc}}$ which contains a maximal subtorus $T_{\text{sc}}$ of $G_{\text{sc}}$, then $B_{\text{sc}}$ determines a system of simple roots $\Delta=\{\alpha_1,...,\alpha_l\}$ for the root system $\Psi$ of $G_{\text{sc}}$ in the character lattice $X^*(T_{\text{sc}})$. We denote the resulting sets of positive roots and negative roots by $\Psi^+$ and $\Psi^-$. Let $B_{\text{sc}}^-$ be the opposite Borel of $B_{\text{sc}}$ so that $B_{\text{sc}}^-\cap B_{\text{sc}}=T_{\text{sc}}$, and let $U^+$ and $U^-$ be the unipotent radicals of $B_{\text{sc}}$ and $B_{\text{sc}}^-$. For a root $\alpha\in \Psi$, we write $U_{\alpha}$ for the root subgroup defined by $\alpha$. Let $T$ be the image of $\widetilde{T}$ in $G$. We identify the set of roots of $G_{\text{sc}}$ with that of $G$ by the natural embedding $X^*(T)\hookrightarrow X^*(T_{\text{sc}})$.

We choose a character $\lambda\colon T_{\text{sc}}\rightarrow \mathbb{G}_m$ that is regular dominant in the sense that $\langle\lambda,\alpha^\vee\rangle\textgreater 0$ for every positive coroot $\alpha^\vee$.

\subsection{Construction of wonderful compactification}\label{reviewofwonderful}
By \cite[Lemma~6.1.1 and Remark~6.1.2]{BrionKumar}, we can fix an algebraic finite-dimensional representation $V_{\lambda}$ of $G_{\text{sc}}$ such that 
\begin{itemize}
    \item the $\lambda$-weight subspace is of dimension 1;
    \item any weight of $T_{\text{sc}}$ other than $\lambda$ in the representation $V_{\lambda}$ is lower than $\lambda$;
    \item for any simple root $\alpha\in\Delta$, the weight space of $\lambda-\alpha$ is nonzero.
\end{itemize}
By the adjointness of $G$ and the regularity of $\lambda$ (see \cite[Lemma~6.1.3]{BrionKumar}), the group $G$ is embedded into  $\mathbb{P}(\End(V_{\lambda}))$ in such a way that the following square commutes:
$$\xymatrix{
G_{\text{sc}} \ar[d] \ar[r]
&\End(V_{\lambda})\ar[d]\\
G_{\text{sc}} \ar@{^{(}->}[r]         &\mathbb{P}(\End(V_{\lambda})).}$$
We define the wonderful compactification $X$ of $G$ to be the closure of $G$ in $\mathbb{P}(\End(V_{\lambda}))$. The $(G\times G)$-action on $G$ naturally extends to $X$. By \cite[Theorem~6.1.8 (iv)]{BrionKumar}, the wonderful compactification $X$ does not depend on the choice of $\lambda$ and $V_{\lambda}$.

\subsection{Big cells of wonderful compactifications}\label{lift}
In order to study the geometry of the wonderful compactification $X$, we fix a basis $v_0, v_1,..., v_n$ of $V_{\lambda}$ consisting of eigenvectors of $T_{\text{sc}}$ such that $v_0$ has weight $\lambda$. Let $(v_i^*)_{0\leq i\leq n}$ be the dual basis of $(v_i)_{0\leq i\leq n}$ in the dual vector space $V_{\lambda}^*$. Consider the open subscheme with set of $k$-rational points 
\begin{equation*}\label{equationofbigcell}
\overline{\Omega}(k):=\left\{f\in X(k)\;\vert\; v_0^*(\widetilde{f}(v_0))\neq 0  \;\text{for some lift $\widetilde{f}$ of $f$ in $\End(V_{\lambda})$}\right\}\subset X(k)\subset \mathbb{P}(\End(V_{\lambda}))(k)
\end{equation*}
which is often referred as the big cell of $X$. By \cite[Proposition~6.1.7]{BrionKumar}, the big cell $\overline{\Omega}$ of $X$ is isomorphic to $U^-\times \overline{T}\times U^+$, where $\overline{T}=\prod_{\Delta}\mathbb{G}_a$. 

By \cite[Theorem~6.1.8 (i)]{BrionKumar}, the $(G\times G)$-translates of the big cell $\overline{\Omega}$ cover $X$ i.e., 
\begin{equation}\label{translationbigcell}
    X=\bigcup_{(g_1, g_2)\in G\times G}(g_1, g_2)\cdot \overline{\Omega}.
\end{equation}

Recall that, by \cite[\S~5.1 Lemma~4]{BLR} or (in a much greater generality) \cite[exposé~VI$_A$, proposition~0.5]{SGA3I}, we have $G= \Omega_G \cdot\Omega_G$, where $\Omega_G\coloneq U^-\times T\times U^+$ is identified with the big cell of $G$ via group multiplication. Moreover, by the proof of the Existence Theorem (see \cite[exposé~XXV]{SGA3III}), $G$ is constructed as a quotient of $\Omega_G\times\Omega_G$. Hence, \Cref{translationbigcell} suggests that the wonderful compactification $X$ might be constructed as a quotient of $G\times\overline{\Omega}\times G$, which will be done in the following sections. Note that in \emph{loc. cit.}, the Existence Theorem for split reductive group schemes is proved assuming that the corresponding result over the complex numbers is known. However, as we will see in the following two sections, the construction of our compactification for split reductive group schemes of adjoint type does not depend on the existence of the classical wonderful compactification over an algebraically closed field.

\section{Rational action on big cells}\label{sectionrationalact}
In this section, we define a rational action of the product of two copies of an adjoint reductive group scheme $G$ on an affine space which models the big cell of a wonderful compactification. Passing to each geometric fiber, this rational action reflects the group action on wonderful compactifications. We then study various properties of this rational action. These properties are crucial for our constructions in \Cref{sectioncompactification}.

\subsection{The setup}
Let $G$ be an adjoint split reductive group scheme over a scheme $S$, and choose a maximal split torus $T$ as in \cite[exposé~XXII, définition~1.13]{SGA3III} and a Borel subgroup $B$ containing $T$. We denote by $e\in G(S)$ the identity section. In this section, if a fiber product is formed over $S$, the lower subscript $S$ will be omitted. Assume that $\text{rk}(T)=l$. We denote the Lie algebras of $T\subset B\subset G$ by $\mathfrak{t}\subset \mathfrak{b}\subset \mathfrak{g}$. Let $(X,\Psi, \Delta, X^{\vee}, \Psi^{\vee}, \Delta^{\vee})$ be the based root datum defined by the Borel $B$, and let $\Psi^+$ (resp., $\Psi^-$) be the set of positive (resp., negative) roots. Then we have the canonical decomposition of $\mathfrak{g}$ into root spaces: $\mathfrak{g}=\bigoplus_{\alpha\in\Psi}\mathfrak{g}_{\alpha}$. We fix an enumeration $\left\{\alpha_1,...,\alpha_l\right\}$ of $\Delta$.

According to \cite[exposé~XXIII, proposition~6.2]{SGA3III}, we can fix a Chevalley system $(X_\alpha\in \Gamma(S,\mathfrak{g}_{\alpha})^{\times})_{\alpha\in\Psi}$ for $G$. Each $X_\alpha$ gives rise to an isomorphism of $S$-groups $$p_\alpha\colon\mathbb{G}_{a,S}\longrightarrow U_\alpha,\;\;\alpha\in\Psi.$$ 
In the following, we will take the coordinate function of $\mathbb{G}_{a,S}$ as coordinate function of $U_{\alpha}$ via the isomorphism $p_{\alpha}$.

We define an $S$-morphism:
\begin{equation}\label{Omegaembedformula}
    \nu\colon\Omega_G:=U^-\times T\times U^+\longrightarrow \overline{\Omega}:=U^-\times \overline{T}\times U^+
\end{equation}
via
$$(u',t,u)\mapsto (u',\prod_{\alpha_i\in \Delta}-\alpha_i(t), u),$$
where $\overline{T}:=\prod_{\Delta}\mathbb{G}_{a,S}.$ In the rest of this paper, the subscheme $\overline{T}$ is canonically identified with the subscheme $\{e\}\times\overline{T}\times\{e\}\subset\overline{\Omega}$, and similarly for $U^-$ and $U^+$.
By the adjointness of $G$, the $\nu\vert_{T}$ is a monomorphism, which is then an open immersion by appealing to \cite[\S~5.3]{oesterlemultiplicatifgroup}, and hence so is $\nu$. In the following, we will implicitly view $\Omega_G$ as a subscheme of $\overline{\Omega}$ via $\nu$. In particular, $T$ acts on $\overline{T}$ via $\nu$. For each simple root $\alpha_i\in \Delta$, we denote by $\mathbb{X}_i$ the coordinate of $\mathbb{G}_{a,S}$ indexed by $\alpha_i$ in $\overline{T}\subset\overline{\Omega}$.

\subsection{The rational action}
We want to define an action of $G\times G$ on $\overline{\Omega}$. In the case of classical wonderful compactification, we know that the action of $G\times G$ does not preserve $\overline{\Omega}$. Hence, instead of defining the action on the entire big cell $\overline{\Omega}$, we seek to take advantage of ``rational action'' which is similar to the notion of a birational group law \cite[exposé~XVIII, définition ~3.1]{SGA3II} (see also \cite[Chapter~4]{BLR}). We require that this ``rational action'' should have large enough domain of definition and coincide on the geometric fibers with the rational $(G\times G)$-action on the classical wonderful compactification. The main result in this part is \Cref{theoremrationalaction}, which is done step by step  by the subsequent lemmas.

\begin{lemma}\label{singlereflectrational}
    Let $\alpha_i\in \Delta$ and let $n_{\alpha_i}=p_{\alpha_i}(1)p_{-{\alpha_i}}(-1)p_{\alpha_i}(1)\in \Norm_G(T)(S)$, which is a representative of the simple reflection $s_i$ defined by $\alpha_i$ in the Weyl group $W$ \cite[exposé~XXII, 3.3]{SGA3III}. Let us fix an enumeration $\{\Psi^-\backslash -\alpha_i, -\alpha_i,\alpha_i, \Psi^+\backslash \alpha_i\}$ of $\Psi$. We consider $V_i\coloneq D_{\overline{\Omega}}(\mathbb{X}_i+xy)$, where $x,y$ are the coordinates of $U_{-\alpha_i}$ and $U_{\alpha_{i}}$ in $\overline{\Omega}$ (which make sense because of the enumeration).  Then there exists a unique morphism $f_i\colon V_i\longrightarrow  \overline{\Omega}$  such that
    \begin{itemize}
        \item[(1)] $f_i$ extends the restriction of the conjugation by $n_{\alpha_i}$ on $G$ to $\Omega_G\bigcap V_i\subset G$;
        \item[(2)] $f_i$ sends $U_{\alpha}$ to $U_{s_i(\alpha)}$, for $\alpha\in\Psi$.
    \end{itemize}

\end{lemma}

\begin{proof}
    By \cite[exposé~XXIII, lemme~3.1.1 (iii)]{SGA3III} and the definition of Chevalley system \cite[exposé~XXIII, définition~6.1]{SGA3III}, for each root $\beta\neq \pm \alpha_i\in \Psi$, we have 
    $$\Ad_{n_{\alpha_i}}(p_\beta(x))=p_{s_i(\beta)}(\epsilon_\beta x),$$ 
    where $\epsilon_\beta=\pm 1$ are given in \emph{loc.~cit.} and $x\in \Gamma(S', \mathcal{O}_{S'})$ for a test $S$-scheme $S'$. By the argument after \cite[exposé~XXIII, définition~6.1]{SGA3III}, the above equality holds for $\beta=\pm \alpha_i$ with $\epsilon_\beta=-1$. For a section
    $$u:=\bigg(\prod_{\gamma\in \Psi^-\backslash \left\{-\alpha_i\right\}}p_\gamma(x_\gamma)\cdot p_{-\alpha_i}(x),\;t,\; p_{\alpha_i}(y)\cdot \prod_{\gamma\in \Psi^+ \backslash\left\{\alpha_i\right\}}p_\gamma(x_\gamma)\bigg)\in  V_i(S') \subset\overline{\Omega}(S'),$$
    where $t=(t_1,...,t_l)\in \overline{T}(S')$ and we choose an order on $\Psi^- \backslash\left\{-\alpha_i\right\}$ and $\Psi^+ \backslash\left\{\alpha_i\right\}$, let 
    $$D\coloneq t_i+xy.$$
    Then we define $f_i$ by sending $u$ to 
    \begin{equation}\label{defoff_i}
        \bigg(\prod_{\gamma\in \Psi^- \backslash\left\{-\alpha_i\right\}}p_{s_i(\gamma)}(\epsilon_\gamma x_\gamma)\cdot p_{-\alpha_i}\left(\frac{-y}{D}\right),\; \alpha_i^\vee(D)\cdot t
    ,\; p_{\alpha_i}\left(\frac{-x}{D}\right)\cdot \prod_{\gamma\in \Psi^+ \backslash\left\{\alpha_i\right\}}p_{s_i(\gamma)}(\epsilon_\gamma x_\gamma)\bigg),
    \end{equation}
    where, in the middle, $\alpha_i^\vee(D)$ acts on $t$ via $\nu$, see \Cref{Omegaembedformula}. 
    
    It is clear that $V_i$ contains $U^-$ and $U^+$.
    The verification of the first condition follows from \Cref{Omegaembedformula} and a computation in the group $G$: assume that $t\in T(S')$, then $D=\alpha_i(t)^{-1}+xy$ and 
    \begin{flalign*}
            &\Ad_{n_{\alpha_i}}\bigg(\prod_{\gamma\in \Psi^- \backslash\left\{-\alpha_i\right\}}p_\gamma(x_\gamma)\cdot p_{-\alpha_i}(x)\cdot t\cdot p_{\alpha_i}(y)\cdot \prod_{\gamma\in \Psi^+ \backslash\left\{\alpha_i\right\}}p_{\gamma}(x_\gamma)\bigg)\\
            =&\prod_{\gamma\in \Psi^- \backslash\left\{-\alpha_i\right\}}p_{s_i(\gamma)}(\epsilon_\gamma x_\gamma)\cdot p_{\alpha_i}(-x)\cdot (t\cdot \alpha_i^{\vee}(\alpha_i(t))^{-1}) \cdot p_{-\alpha_i}(-y)\cdot \prod_{\gamma\in \Psi^+ \backslash\left\{\alpha_i\right\}}p_{s_i(\gamma)}(\epsilon_\gamma x_\gamma)\\
             =&\prod_{\gamma\in \Psi^- \backslash\left\{-\alpha_i\right\}}p_{s_i(\gamma)}(\epsilon_\gamma x_\gamma)\cdot p_{\alpha_i}(-x)\cdot p_{-\alpha_i}(-\alpha_i(t)y)\cdot (t\cdot \alpha_i^{\vee}(\alpha_i(t))^{-1}) \cdot \prod_{\gamma\in \Psi^+ \backslash\left\{\alpha_i\right\}}p_{s_i(\gamma)}(\epsilon_\gamma x_\gamma)\\
             =&\prod_{\gamma\in \Psi^- \backslash\left\{-\alpha_i\right\}}p_{s_i(\gamma)}(\epsilon_\gamma x_\gamma)\cdot p_{-\alpha_i}\left(\frac{-y}{D}\right)\cdot \alpha_i^{\vee}(1+\alpha_i(t)xy)\cdot p_{\alpha_i}\left(\frac{-x}{1+\alpha_i(t)xy}\right)\cdot \\&(t\cdot \alpha_i^{\vee}(\alpha_i(t))^{-1}) \cdot \prod_{\gamma\in \Psi^+ \backslash\left\{\alpha_i\right\}}p_{s_i(\gamma)}(\epsilon_\gamma x_\gamma)\\
            =&\prod_{\gamma\in \Psi^- \backslash\left\{-\alpha_i\right\}}p_{s_i(\gamma)}(\epsilon_\gamma x_\gamma)\cdot p_{-\alpha_i}\left(\frac{-y}{D}\right)\cdot (t\cdot \alpha_i^{\vee}(D))
    \cdot p_{\alpha_i}\left(\frac{-x}{D}\right)\cdot \prod_{\gamma\in \Psi^+ \backslash\left\{a\right\}}p_{s_i(\gamma)}(\epsilon_\gamma x_\gamma),
    \end{flalign*} 
    where \cite[exposé~XXII, définitions~1.5 (b)]{SGA3III} is used in the first equality and \cite[exposé~XXII, notations~1.3]{SGA3III} is used in the third equality. The uniqueness follows from (1) and the schematic density of $\Omega_G$ in $\overline{\Omega}$, in particular, the definition of $f_i$ does not depend on the choice of the order on $\Psi^- \backslash\left\{-\alpha_i\right\}$ and $\Psi^+ \backslash\left\{\alpha_i\right\}$ 
\end{proof}

We choose an (not necessarily reduced) expression for the longest element $w_0$ in the Weyl group $W$: $w_0=s_{i_m}\cdot s_{i_{m-1}}\cdot...\cdot s_{i_1}$, where $m\in\mathbb{Z}_{>0}$ is greater or equal to the length of $w_0$, and $1 \leq i_1,..., i_m\leq l$. Set $n_0\coloneq n_{\alpha_{i_m}}\cdot n_{\alpha_{i_{m-1}}}...\cdot n_{\alpha_{i_1}}\in\Norm_G(T)(S)$ with $n_{\alpha_i}\coloneq p_{\alpha_i}(1)p_{-\alpha_i}(-1)p_{\alpha_i}(1)$ for all $1\leq i\leq l$. Then $n_0\in\Norm_G(T)(S)$ is a representative of $w_0$.

\begin{lemma}\label{w_0invariantsub}
    There exist two $S$-rational morphisms
    $$f\colon \overline{\Omega}\dashrightarrow \overline{\Omega}\quad\text{and} \quad f'\colon \overline{\Omega}\dashrightarrow \overline{\Omega}$$
    such that
    \begin{itemize}  
    \item[(1)] if the base scheme $S$ is the spectrum of a strict henselian local ring, there exist $u_0^+\in ~U^+(S)$ and $u_0^-\in U^-(S)$ such that $(u_0^-, (0)_{1\leq i\leq l}, u_0^+)\in \Dom(f)(S)\bigcap f^{-1}(\Dom(f')(S))$;
    
    \item[(2)] $f$ (resp., $f'$) extends the conjugation by $n_0$ (resp., $n_0^{-1}$) on $\Omega_G\bigcap \Dom(f)$ and $\Omega_G\bigcap \Dom(f')$;          
    \end{itemize}
\end{lemma}

\begin{proof}
    We set
    $f=f_{i_m}...f_{i_1}$ where $f_i$ is defined in \Cref{singlereflectrational} for any $1\leq i\leq l$ and is viewed as an $S$-rational automorphism of $\overline{\Omega}$. 
    Note that $n_0^{-1}=n_{\alpha_{i_1}}^3\cdot n_{\alpha_{i_{2}}}^3\cdot...\cdot n_{\alpha_{i_m}}^3$ because of the equalities $n_{\alpha_i}^4=e$ in \cite[exposé~XX, théorème~3.1 (v)]{SGA3III}. Hence we define $f'$ in a similar way. 

    For each simple root $\alpha_i\in\Delta$, we fix an enumeration of $\Psi$ and the coordinates $x$ and $y$ of $U_{-\alpha_{i}}$ and $U_{\alpha_{i}}$ as in \Cref{singlereflectrational}.
    Let $\underline{V_j}\coloneq D_{U^-\times U^+}(xy)\subset U^-\times U^+ $. Then, for all $1\leq j\leq l$, after restricting to $U^-\times (0)_{1\leq i\leq l}\times U^+$ which is identified with $U^-\times U^+$, the $f_j$ in \Cref{singlereflectrational} becomes
    an isomorphism
    \begin{equation}
    \underline{f_j}\colon \underline{V_j}\longrightarrow \underline{V_j}
    \end{equation}
which sends
$$\bigg(\prod_{\gamma\in \Psi^- \backslash\left\{-\alpha_j\right\}}p_\gamma(x_\gamma)\cdot p_{-\alpha_j}(x),\; p_{\alpha_j}(y)\cdot \prod_{\gamma\in \Psi^+ \backslash\left\{\alpha_j\right\}}p_\gamma(x_\gamma)\bigg)\in  \underline{V_j}(S')$$
to 
$$ \bigg(\prod_{\gamma\in \Psi^- \backslash\left\{-\alpha_j\right\}}p_{s_j(\gamma)}(\epsilon_\gamma x_\gamma)\cdot p_{-\alpha_j}\left(\frac{-1}{x}\right),\; p_{\alpha_j}\left(\frac{-1}{y}\right)\cdot \prod_{\gamma\in \Psi^+ \backslash\left\{\alpha_j\right\}}p_{s_j(\gamma)}(\epsilon_\gamma x_\gamma)\bigg).$$
By intersecting $\underline{V_{i_m}}$ with $\underline{V_{i_{m-1}}}$ and taking the inverse image along the isomorphism $\underline{f_{i_{m-1}}}$, we get an $S$-dense open subscheme of $\underline{V_{i_{m-1}}}$. By further intersecting the open subscheme with $\underline{V_{i_{m-2}}}$ and taking the inverse image along $\underline{f_{i_{m-2}}}$, we get an $S$-dense open subscheme of $\underline{V_{i_{m-2}}}$. Continuing this process, after finitely many steps, we end up with an $S$-dense open subscheme $\underline{V_0}\subset \underline{V_{i_1}}$.
Hence, if we take $(u_0^-,u_0^+)$ to be an $S$-rational point in $\underline{V_0}$, then $(u_0^-,(0)_{1\leq i\leq l},u_0^+)\in \Dom(f)(S)$. Under the assumption of (1), the existence of such a section $(u_0^-,u_0^+)$ is ensured by \cite[\S~2.3, Proposition~5]{BLR}. By a similar argument and further shrinking, we can also arrange that $(u_0^-,(0)_{1\leq i\leq l},u_0^+)\in f^{-1}(\Dom(f')(S))$. Then (1) follows. The claim (2) follows from \Cref{singlereflectrational} (1).
\end{proof}

The $f$ and $f'$ in \Cref{w_0invariantsub} is used to swap the positive and the negative root subgroups through the middle $\overline{T}$, which is an important operation in the following result, even if it is merely defined over an open subscheme.

\begin{lemma}\label{extensionofmultiplication}
    Consider the $S$-rational morphism 
    $$\Theta: U^+\times \overline{T}\times U^-\dashrightarrow \overline{\Omega}$$
    given by the group multiplication $U^+\times T\times U^-\longrightarrow G$ and the open immersion $\nu$ \Cref{Omegaembedformula} where $\Omega_G$ is identified with the big cell of $G$ via the group multiplication. Then $\Dom(\Theta)$ contains $e\times \overline{T}\times e$.
\end{lemma}

\begin{proof} For simplicity, we denote $(0)_{1\leq i\leq l}\in \overline{T}(S)$ by $0$.

    First of all, by flat descent \cite[exposé~XVIII, proposition~1.6]{SGA3II} and a limit argument, we can assume that the base $S$ is the spectrum of a strictly henselian local ring.
    Choose a test $S$-scheme $S'$.
    Note that by the definition of $\Theta$, $\Dom(\Theta)$ already contains $e\times T\times e$, where $e\in G(S)$ is the identity section. Hence it suffices to show that $e\times (t_i)_{1\leq i\leq l}\times e\subset \Dom(\Theta)$ for $(t_i)_{1\leq i\leq l}\in \overline{T}(S')\backslash T(S')$. For this, we will first construct an $S$-rational morphism $\Theta':U^+\times \overline{T}\times U^-\dashrightarrow \overline{\Omega}$ such that $(e,0,e) \subset \Dom(\Theta')$ and $\Theta=\Theta'$ as $S$-rational morphisms.

    We fix a section $(u_0^-, 0, u_0^+)\in \overline{\Omega}(S)$ as in \Cref{w_0invariantsub} ~(1).
    For any test $S$-scheme $S'$ and a section $(u^+, \overline{t}, u^-)\in (U^+\times \overline{T}\times U^-)(S')$, we define $\Theta'(u^+, \overline{t}, u^-)$ step by step and along the way we shrink the definition domain of $\Theta'$. 

 
    As the \emph{first step}, by restricting to nonempty open subschemes of $U^+$ and $U^-$, we can have 
    $$u^+(u_0^-)^{-1}=(\dot{u}^-,\dot{t}, \dot{u}^+)\in \Omega_G(S'),$$
    $$(u_0^+)^{-1}u^-=(\dot{v}^-,\dot{s}, \dot{v}^+)\in \Omega_G(S').$$

    
    As the \emph{second step}, we  let  
    $$f(\dot{t}u_0^-\dot{t}^{-1}, \dot{t}\overline{t}\dot{s}, \dot{s}^{-1}u_0^+\dot{s})=(\widetilde{u}^-, \widetilde{t}, \widetilde{u}^+).$$
    We further restrict to an open subscheme $D'\subset U^+\times\overline{T}\times U^-$ so that the section $$f'((n_0(\dot{t}\dot{u}^+\dot{t}^{-1})n_0^{-1})\widetilde{u}^-,\;\widetilde{t},\;\widetilde{u}^+(n_0(\dot{s}^{-1}\dot{v}^-\dot{s})n_0^{-1})) \in\overline{\Omega}(S')$$
is well-defined, where $n_0$ is defined in \Cref{w_0invariantsub}, and we denote this section by 
$(\widetilde{u}^-_0,\widetilde{t}_0,\widetilde{v}^+_0).$
    
    
    As the \emph{third step}, we define
    $$\Theta'(u^+, \overline{t}, u^-)=(\dot{u}^-\widetilde{u}^-_0,\widetilde{t}_0,\widetilde{v}^+_0\dot{v}^+)\in \overline{\Omega}(S').$$

    The fact that $\Theta'\vert_{U^+\times T\times U^-}$ coincides with the group multiplication of $G$ follows from the following computation in the group $G$ when $\overline{t}\in T(S')$:
    \begin{equation*}
       \begin{aligned}
        \Theta'(u^+, \overline{t}, u^-) 
        &=\dot{u}^-\widetilde{u}^-_0\widetilde{t}_0\widetilde{v}^+_0\dot{v}^+\\
        &= \dot{u}^-f'((n_0(\dot{t}\dot{u}^+\dot{t}^{-1})n_0^{-1})\widetilde{u}^-,\;\widetilde{t},\;\widetilde{u}^+(n_0(\dot{s}^{-1}\dot{v}^-\dot{s})n_0^{-1}))\dot{v}^+\\
        &= \dot{u}^- (\dot{t}\dot{u}^+\dot{t}^{-1})n_0^{-1}\widetilde{u}^-\widetilde{t}\widetilde{u}^+n_0(\dot{s}^{-1}\dot{v}^-\dot{s})  \dot{v}^+\\
        &= \dot{u}^- (\dot{t}\dot{u}^+\dot{t}^{-1})n_0^{-1}f(\dot{t}u_0^-\dot{t}^{-1}, \dot{t}\overline{t}\dot{s}, \dot{s}^{-1}u_0^+\dot{s})n_0(\dot{s}^{-1}\dot{v}^-\dot{s})  \dot{v}^+\\
        &= \dot{u}^- (\dot{t}\dot{u}^+\dot{t}^{-1})\dot{t}u_0^-\dot{t}^{-1}\dot{t}\overline{t}\dot{s} \dot{s}^{-1}u_0^+\dot{s}(\dot{s}^{-1}\dot{v}^-\dot{s})  \dot{v}^+\\
        &= \dot{u}^- \dot{t}\dot{u}^+u_0^-\overline{t}u_0^+\dot{v}^-\dot{s}\dot{v}^+
        =u^+\overline{t}u^-,
       \end{aligned}
    \end{equation*}
    where \Cref{w_0invariantsub}~(2) is used in the third, fifth equalities. 
    Hence we have $\Theta'=\Theta$ as $S$-rational morphisms. By the choice of $u_0^-$ and $u_0^+$ and \Cref{w_0invariantsub}~(1), we have $$ (e,0,e)\in D'\subset \Dom(\Theta).$$

    Now, for any $S$-scheme $S'$ and $t\in T(S')$, we consider the following diagram of $S'$-rational morphisms
    
\begin{equation}\label{diagramtheta}
\begin{gathered}
    \xymatrix{
U^+_{S'}\times \overline{T}_{S'}\times U^-_{S'}\ar@{-->}[r]^-{\Theta_{S'}}\ar[d]^{M(t)} & \overline{\Omega}_{S'}\ar[d]^{M'(t)} &\\
U^+_{S'}\times \overline{T}_{S'}\times U^-_{S'}\ar@{-->}[r]^-{\Theta_{S'}} &  \overline{\Omega}_{S'}, & }
\end{gathered}
\end{equation}
where $M(t)$ sends a section $(v^+, s, v^-)\in (U^+\times \overline{T}\times U^-)(S')$ to $(tv^+t^{-1},ts, v^-)$ and $M'(t)$ sends a section $(u^-, w, u^+)\in\overline{\Omega}(S')$ to $(tu^-t^{-1}, tw, u^+)$. The above diagram is commutative because so is it after intersecting with $G_{S'}$. Therefore $\Dom(\Theta_{S'})$ is stable under the $S'$-endomorphism $M(t)$ for any $t\in T(S')$. Combining with $(e,0,e)\in \Dom(\Theta_{S'})$, we conclude $e\times \overline{T}\times e\subset \Dom(\Theta)$. 
\end{proof}

Now we are in the position to define an $S$-rational action of $G\times G$ on the big cell $\overline{\Omega}$.
The following theorem is, in spirit, similar to \cite[exposé~XXV, proposition~2.9]{SGA3III}, and we base our proof on theirs.

\begin{theorem}\label{theoremrationalaction}
    There exists a unique $S$-rational morphism $\pi\colon G\times\overline{\Omega}\times G\dashrightarrow \overline{\Omega}$ whose definition domain $\Dom(\pi)$ contains $e\times\overline{\Omega}\times e$ 
    such that 

    (1) The restriction of $\pi$ to the open subscheme $(G\times\Omega_G\times G)\bigcap\Dom(\pi)$ is given by 
$$(g_1,g,g_2)\longmapsto g_1gg_2^{-1}.$$
        
    (2) For every $x\in S$, the geometric fiber $\pi_{\overline{k(x)}}$ agrees with the restriction of the rational action of $G_{\overline{k(x)}}\times G_{\overline{k(x)}}$ on the big cell $\overline{\Omega}_{\overline{k(x)}}$ of the wonderful compactification of $G_{\overline{k(x)}}$ over the algebraically closed field $\overline{k(x)}$.

    (3) The restriction $\pi\vert_{e\times\overline{\Omega}\times e}$ is the projection onto the middle factor. 
\end{theorem}

\begin{remark}
    In the case when $S$ is the spectrum of an algebraically closed field, if we assume the existence of the classical wonderful compactification, the desired $S$-rational morphism $\pi$ in \Cref{theoremrationalaction} exists simply as a restriction of the group action morphism of the wonderful compactification of $G$.  
\end{remark}

\begin{proof}
    Note that (2) follows from (1) by the $S$-density of $\Omega_G$ in $\overline{\Omega}$. The \emph{uniqueness} of $\pi$ also follows from this reason and (1).

    For the \emph{existence} of $\pi$, since $\Omega_G$ is $S$-dense in $G$, it suffices to define an $S$-rational morphism $$\pi\colon\Omega_G\times\overline{\Omega}\times\Omega_G\dashrightarrow \overline{\Omega}$$
    which is defined over an open subscheme $\mathcal{R}$ containing $e\times\overline{\Omega}\times e$.
    We will define our $\pi$ step by step and shrink $\Omega_G\times\overline{\Omega}\times\Omega_G$ along the way to obtain the sought $\mathcal{R}$. 
    Let $S'$ be a test $S$-scheme.
    Consider an element
    $$A:=((u_1^-,t_1, u_1^+),(u^-,t, u^+),(u_2^-,t_2, u_2^+))\in (\Omega_G\times\overline{\Omega}\times\Omega_G)(S').$$
    
    
    As the \emph{first step}, we let 
    $$V\coloneq\sigma^{-1}(\Omega_G)\bigcap \Omega_G$$
    where $\sigma$ is the inverse morphism of the group scheme $G$. Then we define an $S$-dense open subscheme 
    $$\mathcal{R}_1\coloneq\Omega_G\times\overline{\Omega}\times  V\subset\Omega_G\times\overline{\Omega}\times\Omega_G.$$
    If $A\in \mathcal{R}_1(S')$, we write 
    $$\sigma(u_2^-,t_2, u_2^+)=(\hat{u}_2^-,\hat{t}_2, \hat{u}_2^+)\in \Omega_G(S').$$
    
    As the \emph{second step}, we define 
    the $S$-dense open subscheme $\mathcal{R}_2\subset\mathcal{R}_1$ given by the following condition:
    $$A\in \mathcal{R}_2\iff
    u_1^+u^-\, \text{and}\,u^+\hat{u}_2^-\, \text{lie in}\; \Omega_G(S').$$
    If so, we write 
    \begin{equation}\label{comutequation1}
        u_1^+u^-=(\dot{u}_1^-, \dot{t}_1, \dot{u}^+)\in\Omega_G(S')
    \end{equation} 
    and 
    \begin{equation}\label{comutequation2}
        u^+\hat{u}_2^-=(\dot{u}^- ,\dot{t}_2 ,\dot{u}_2^+)\in\Omega_G(S').
    \end{equation}


   As the \emph{third step}, we further restrict to the open subscheme $\mathcal{R}\subset\mathcal{R}_2$ defined by the following:
    $$A\in \mathcal{R}\iff(\dot{t}_1\dot{u}^+\dot{t}_1^{-1},\dot{t}_1 t_0\dot{t}_2, \dot{t}_2^{-1}\dot{u}^-\dot{t}_2)\in \Dom(\Theta)(S')$$
    where $\Theta$ is defined in \Cref{extensionofmultiplication}. The open subscheme $\mathcal{R}$ is $S$-dense because, by \Cref{extensionofmultiplication}, $e\times \overline{T}\times e\subset \Dom(\Theta)$.
    If $A\in\mathcal{R}(S')$, we write 
    \begin{equation}\label{comutequation3}    \Theta(\dot{t}_1\dot{u}^+\dot{t}_1^{-1},\dot{t}_1 t\dot{t}_2,\dot{t}_2^{-1}\dot{u}^-\dot{t}_2)=(\Ddot{u}^-, \Ddot{t}, \Ddot{u}^+)\in\overline{\Omega}(S').
    \end{equation}
    Now we have all components of $\pi(A)$ on $U^-$, $U^+$ and $\overline{T}$ in the ``right'' places. Hence 
    we define the image of $A$ under $\pi$ to be 
    \begin{equation}\label{definitionofpi}
        (u_1^- (t_1\Dot{u}_1^- t_1^{-1})(t_1\Ddot{u}^-t_1^{-1}),t_1\Ddot{t} \hat{t}_2,(\hat{t}_2^{-1}\Ddot{u}^+\hat{t}_2)(\hat{t}_2^{-1}\dot{u}_2^+\hat{t}_2)\hat{u}_2^+)\in \overline{\Omega}(S').
    \end{equation}
     Note that, when $(u^-_1, t_1,u^+_1)=e$ and $(u^-_2, t_2,u^+_2)=e$, we have that $\dot{u}^+=e$ and $\dot{u}^-=e$. Then, by \Cref{extensionofmultiplication} (1), we conclude that $\mathcal{R}$ contains $e\times \overline{\Omega}\times e$. The claim (4) follows from the definition of $\pi$ and $\mathcal{R}$ above and \Cref{extensionofmultiplication} (2).

    To show the claim (1), according to the definition of $\pi$, if $A\in(\Omega_G\times\Omega_G\times\Omega_G)(S')\bigcap\mathcal{R}(S')$, we have that
    \begin{equation}\label{comutequation4}
        \pi(A)=u_1^-t_1\dot{u}_1^-\Ddot{u}\Ddot{t}\Ddot{u}^+\dot{u}_2^+\hat{t}_2\hat{u}_2^+
    \end{equation}
    holds in $G(S')$.
    Notice that \Cref{comutequation1} -- \eqref{comutequation3} give rise to 
    $$u_1^+u^-=\dot{u}_1^-\dot{t}_1\dot{u}^+,\;\;u^+\hat{u}_2^-=\dot{u}^-\dot{t}_2 \dot{u}_2^+\;\;$$
    $$\;\;\dot{t}_1\dot{u}^+\dot{t}_1^{-1}\dot{t}_1t\dot{t}_2\dot{t}_2^{-1}\dot{u}^-\dot{t}_2=\Ddot{u}^- \Ddot{t}\Ddot{u}^+$$
    in $G(S')$,
    where \Cref{extensionofmultiplication} is used to deduce the last equation. 
    Also since $\sigma$ is the inverse operation, combining with the first step, we have 
    $$\hat{u}_2^-\hat{t}_2\hat{u}_2^+=(u_2^-t_2u_2^+)^{-1}.$$
    Substituting the above four formulas into  \Cref{comutequation4} in turn, after a computation, we have 
    $$\pi(A)=(u_1^-t_1 u_1^+)(u^-tu^+)(u_2^-t_2u_2^+)^{-1},$$
    as desired. The claim (3) follows from the definition of $\pi$.
\end{proof}

\subsection{Properties of the rational action}
We are going to show several properties of $\pi$. These properties are analogues of the conditions which are imposed on a birational group law to ensure the existence of a group scheme solution to the group law.

The following lemma says that the rational action $\pi$ of $ G\times G$ on $\overline{\Omega}$ defined in  \Cref{theoremrationalaction} is associative.

\begin{lemma}\label{associativity}
    The following two $S$-rational morphisms
    $$\phi\colon(G\times G)\times(G\times\overline{\Omega}\times G)\dashrightarrow \overline{\Omega},$$
    $$((g_1,g_2),(g_1', x, g_2'))\mapsto\pi(g_1, \pi(g_1', x, g_2'),g_2)$$
    and 
    $$\psi\colon(G\times G)\times(G\times\overline{\Omega}\times G)\dashrightarrow \overline{\Omega},$$
    $$((g_1,g_2),(g_1', x, g_2'))\mapsto \pi(g_1g_1', x, g_2 g_2')$$
    coincide.
\end{lemma}

\begin{proof}
    By (1) of \Cref{theoremrationalaction}, $\phi$ and $\psi$ coincide over an $S$-dense open subscheme.
\end{proof}

The $S$-rational morphism $\pi$ in \Cref{theoremrationalaction} defines a (relative) rational morphism in the sense of \cite[exposé~XVIII, définition~1.5]{SGA3II}, which can be viewed as a rational action of $ G\times G$ on $\overline{\Omega}$. In the following lemma, we prove that this rational action is ``birational'' and ``strict'', which is an analogue of the notion of a strict birational group law, see \cite[Chapter~5, Section~5.2, Definition~1]{BLR}. We adopt the strategy of the proof of \cite[Section~6, Proposition~6.3.13]{redctiveconrad} into our context.

\begin{lemma}\label{strictness}
    The $S$-rational morphisms
    $$\Phi\colon  G\times \overline{\Omega}\times G\dashrightarrow
    \ G\times \overline{\Omega}\times G,$$
    $$(g_1, x, g_2)\mapsto(g_1, \pi(g_1, x, g_2), g_2),$$
    $$\Psi\colon  G\times \overline{\Omega}\times G\dashrightarrow
     G\times \overline{\Omega}\times G,$$
    $$(g_1, x, g_2)\mapsto(g_1, \pi(g_1^{-1}, x, g_2^{-1}), g_2)$$
    are birational. Moreover, there exists an $\overline{\Omega}$-dense open subscheme $\mathcal{V}\subset \Dom(\Phi)$ containing $\{e\}\times\overline{\Omega}\times\{e\}$ such that the restriction $\Phi\vert_{\mathcal{V}}$ is an open immersion into $ \Dom(\Psi)$.
\end{lemma}

\begin{proof}
    The rational morphisms $\Phi$ and $\Psi$ are birational because, by \Cref{associativity}, they are inverse to each other. 
    Let $\mathcal{U}$ and $\mathcal{U}'$ be the definition domains of $\Phi$ and $\Psi$. 
    Let $$\mathcal{V}:=\Phi^{-1}(\mathcal{U'})\bigcap \mathcal{U}\;\; \text{and}\;\; \mathcal{V}':=\Psi^{-1}(\mathcal{U})\bigcap \mathcal{U}'.$$
    Note that $\Phi\vert_{\mathcal{V}}$ and $\Psi\vert_{\mathcal{V}'}$ are inverse to each other and by \Cref{theoremrationalaction}, $\mathcal{V}$ and $\mathcal{V}'$ contain $\{e\}\times\overline{\Omega}\times\{e\}$.
\end{proof}

\section{Compactification for split adjoint reductive group schemes}\label{sectioncompactification}
We shall keep the notations of \Cref{sectionrationalact}. In this section, we first define a compactification of the split adjoint reductive group $G$ as an fppf-sheaf over the category $\Sch/S$. Then we prove that this sheaf is in fact a projective scheme equipped with a $(G\times G)$-action which equivariantly contains $G$ as an open $S$-dense subscheme. Moreover, we show that geometric fibers of this compactification are wonderful compactifications as we recalled in \Cref{wonderfulcomp}. 

\subsection{An equivalence relation}
Recall from \Cref{theoremrationalaction} that we have the $S$-rational morphism $\pi\colon G\times\overline{\Omega}\times G\dashrightarrow \overline{\Omega}$. The following construction is inspired by \cite[exposé~XVIII, 3.2.3]{SGA3II}. 

\begin{definition}\label{relationonOmega}
     We define a relation on $G\times\overline{\Omega}\times G$ as follows:
     for a test $S$-scheme $S'$ and $(g_1, x, g_2)$, $(g_1', x', g_2')\in( G\times\overline{\Omega}\times G )(S')$, we say that $(g_1, x, g_2)$ and $(g_1', x', g_2')$ are equivalent, if and only if, there exist an fppf cover $S''\rightarrow S'$ and a section $(a_1, a_2)\in ( G\times G)(S'')$   such that $\pi(a_1g_1,x,a_2g_2)$ and $\pi(a_1g_1',x',a_2g_2')$ are both well defined and are equal over $S''$. We write $(g_1, x, g_2)\sim(g_1', x', g_2')$ if  $(g_1, x, g_2)$, $(g_1', x', g_2')\in ( G\times\overline{\Omega}\times G)(S')$ are equivalent.
\end{definition}

Our next goal is to show that the relation in \Cref{relationonOmega} is an equivalence relation. As a preparation, we need the following lemma which is inspired by \cite[exposé~XVIII, lemme~3.3]{SGA3II}. This lemma shows that the equivalence of any two sections of $G\times \overline{\Omega}\times G$ defined in \Cref{relationonOmega} can be tested with respect to any section $(a_1,a_2)$ bringing the two sections into $\Dom(\pi)$.

\begin{lemma}\label{choiceofequivalence}
    Consider two sections 
    $$(g_1, x, g_2),(g_1', x', g_2')\in( G\times\overline{\Omega}\times G)(S)$$
    together with $(a_1, a_2)\in (G\times G)(S')$, where $S'$ is an fppf cover of $S$, such that 
    $$\pi(a_1g_1, x, a_2g_2)=\pi(a_1g_1', x', a_2g_2').$$
    Then, for $(h_1, h_2)\in (G\times G)(S'')$ where $S''$ is an $S$-scheme, if $\pi(h_1g_1, x, h_2g_2)$ and $\pi(h_1g_1', x', h_2g_2')$ are well defined, then they are equal.
\end{lemma}

\begin{proof} 
    We have that
    \begin{flalign*}
         \pi(z_1g_1, x, z_2g_2) &=\pi(z_1a_1^{-1},\pi(a_1g_1,x,a_2g_2),z_2a_2^{-1})\\         &=\pi(z_1a_1^{-1},\pi(a_1g_1',x',a_2g_2'),z_2a_2^{-1})\\
         &=\pi(z_1g_1', x', z_2g_2')
    \end{flalign*}
    holds as an equality of $S'$-rational morphisms from $(G\times G)_{S'}$ to $\overline{\Omega}_{S'}$ with respect to the variable $(z_1,z_2)$, where \Cref{associativity} is used in the first and the last equality. Hence by \cite[exposé~XVIII, proposition~1.6]{SGA3II}, 
    $\pi(z_1g_1, x, z_2g_2)=\pi(z_1g_1', x', z_2g_2')$ holds as an equality of $S$-rational morphisms from $G\times G$ to $\overline{\Omega}$. Therefore,
    $\pi(h_1g_1, x, h_2g_2)=\pi(h_1g_1', x', h_2g_2').$\qedhere
\end{proof}

\begin{lemma}\label{equivalencerelation}
    The relation in \Cref{relationonOmega} is an equivalence relation.
\end{lemma}

\begin{proof}
    It is clear that this relation is symmetric. 
    
    To show the reflexivity, for an $S$-scheme $S'$ and $(g_1,x,g_2)\in ( G\times\overline{\Omega}\times G) (S')$, 
    we take the section $(a_1,a_2)\coloneq(g_1^{-1},g_2^{-1})\in (G\times G)(S')$. Then, $\pi(a_1g_1,x, a_2g_2)=\pi(e,x,e)$ is well-defined by \Cref{theoremrationalaction}. This proves the reflexivity.

    To show the transitivity, we assume that, over $S'$, we have $$(g_1, x, g_2)\sim(g_1', x', g_2'),\;(g_1', x', g_2')\sim(g_1'', x'', g_2'').$$ 
         By \Cref{strictness} and \cite[exposé~XVIII, proposition~1.1 (i)]{SGA3II}, there exists an $S$-dense open subscheme $\mathcal{V}$ of $G\times G$ such that for any section $(h_1,h_2)$ of $\mathcal{V}$ valued in some $S$-scheme, $\pi(h_1g_1, x, h_2g_2)$, $\pi(h_1g_1', x', h_2g_2')$ and $\pi(h_1g_1'', x'', h_2g_2'')$ are all well defined. 
         Such an fppf local section $(h_1,h_2)$ exists by \cite[exposé~XVIII, proposition~1.7]{SGA3II}.
         Then the transitivity follows from \Cref{choiceofequivalence}.
\end{proof}

\begin{lemma}\label{graphlemma}
    Consider the following $S$-rational morphism
     \begin{flalign*}        \phi\colon( G\times\overline{\Omega}\times G)\times( G\times G)&\dashrightarrow \overline{\Omega},\\         
     (g_1, x ,g_2,g_1',g_2')&\longmapsto\pi(g_1'^{-1},\pi(g_1,x,g_2), g_2'^{-1}).
    \end{flalign*}
    Let $\mathcal{U}$ be the definition domain of the rational morphism $\phi$, and let $\Gamma$ be the graph of the morphism $\phi\vert_{\mathcal{U}}$. 
    We consider a section 
    $$(g_1, x ,g_2,g_1',g_2',y)\in (( G\times\overline{\Omega}\times G)\times( G\times G)\times\overline{\Omega})(S).$$
    If $(g_1, x ,g_2,g_1',g_2',y)\in \Gamma(S)$, then $(g_1, x, g_2)\sim(g_1', y, g_2').$ 
\end{lemma}

\begin{proof}
    Assume that $(g_1, x ,g_2,g_1',g_2',y)\in \Gamma(S)$, i.e., $y=\phi(g_1, x ,g_2,g_1',g_2')$. Note that $\pi(g_1, x ,g_2)$ may not be well defined because $\pr_{123}(\mathcal{U})$ is possibly larger than the definition domain of $\pi$. However, by \Cref{theoremrationalaction} and \cite[exposé~XVIII, proposition~1.7]{SGA3II}, there exists an $(h_1,h_2)\in(G\times G)(S')$, where $S'$ is an fppf cover of $S$, such that $\pi(h_1g_1, x ,h_2g_2)$ and $\pi(h_1g_1', y ,h_2g_2')$ are both well defined. Again, by the associativity of $\pi$ (\Cref{associativity}) and the definition of $\phi$ and by viewing $(g_1, x ,g_2,g_1',g_2')$ as variables,  we have the following equalities of $S'$-rational morphisms from $( G\times \overline{\Omega}\times G \times G\times G )_{S'}$ to $\overline{\Omega}_{S'}$ 
    \begin{flalign*}
         \pi(h_1g_1', y ,h_2g_2')&=\pi(h_1g_1', \phi(g_1, x ,g_2,g_1',g_2') ,h_2g_2')\\ &=\pi(h_1, \pi(g_1',\phi(g_1, x ,g_2,g_1',g_2') ,g_2') ,h_2)\\         
         &=\pi(h_1,\pi(g_1,x,g_2),h_2)
         =\pi(h_1g_1, x, h_2g_2).
    \end{flalign*}
     It follows that $\pi(h_1g_1', y ,h_2g_2')=\pi(h_1g_1, x, h_2g_2)$ holds as an equality in $\overline{\Omega}(S')$.
\end{proof}

\subsection{Schematic nature of the compactification}

Now we are in the position to construct a compactification of the group $G$ as a sheaf. The following definition is inspired by the definition given in \cite[exposé~XVIII, page~387]{SGA3II}.

\begin{definition}\label{defintionofcompactification}
    Thanks to \Cref{equivalencerelation}, we define the compactification of the adjoint group scheme $G$ to be the quotient sheaf $\mathcal{X}$ of the scheme $ G\times\overline{\Omega}\times G$ with respect to the equivalence relation constructed in \Cref{equivalencerelation}. We denote the quotient morphism by $ Q_G\colon G\times\overline{\Omega}\times G\twoheadrightarrow\mathcal{X}$.
\end{definition}

\begin{remark}
    \Cref{defintionofcompactification} is compatible with any base change in the sense that, for any morphism $S'\longrightarrow S$, the fppf-sheaf $\mathcal{X}_{S'}$ over the category $\Sch/S'$ is isomorphic to the quotient sheaf of $ G_{S'}\times_{S'}\overline{\Omega}_{S'}\times_{S'} G_{S'}$ obtained by applying \Cref{defintionofcompactification} to the $S'$-rational morphism $\pi_{S'}$.
\end{remark}

We adapt the method due to Artin in the proof of \cite[exposé~XVIII, proposition~3.5]{SGA3II} to deduce the following result which says that the equivalence relation of  \Cref{equivalencerelation} can be encoded in a nice scheme.

\begin{theorem}\label{algebraicspace}
    The quotient sheaf of the relation $\Gamma\rightrightarrows G\times\overline{\Omega}\times G$ is isomorphic to $\mathcal{X}$, where $\Gamma$ is defined in \Cref{graphlemma}, the two morphisms are $\pr_{123}$ and $\pr_{456}$ and we implicitly identify the target of the projection $\pr_{456}$ with $ G\times\overline{\Omega}\times G$ by switching the last two coordinates.
\end{theorem}

\begin{proof}
    We will verify that two sections of $G\times \overline{\Omega}\times G$ are equivalent if and only if they come from a section of $\Gamma$.
    As in \Cref{graphlemma}, let $\mathcal{U}$ be the definition domain of $$\phi\colon(G\times\overline{\Omega}\times G)\times(G\times G)\dashrightarrow \overline{\Omega},$$
    and let $\Gamma$ be the graph of $\phi\vert_{\mathcal{U}}$. Then we have the natural immersion
$$\Gamma\hookrightarrow(G\times\overline{\Omega}\times G)\times (G\times G)\times \overline{\Omega}.$$

   We fix a test $S$-scheme $S'$ and two sections
    $(g_1,x,g_2), (g_1',y,g_2')\in( G\times\overline{\Omega}\times G)(S')$. Since the definition domain of $\phi_{S'}$ certainly contains $\mathcal{U}\times S'$, by applying \Cref{graphlemma} to $\phi_{S'}$, it follows that $(g_1,x,g_2,g_1',g_2',y)\in\Gamma(S')$ implies that $(g_1,x,g_2)\sim (g_1',y,g_2')$.
    
    Conversely, suppose that, for a test $S$-scheme $S'$, 
    $$(g_1,x,g_2),\; (g_1',y,g_2')\in( G\times\overline{\Omega}\times G)(S')\; \text{and} \; (g_1,x,g_2)\sim (g_1',y,g_2'),$$
    i.e., that there exist an fppf cover $S''$ over $S'$ and a section $(a_1,a_2)\in ( G\times G)(S'')$ such that
$\pi(a_1g_1,x,a_2g_2)=\pi(a_1g_1',y,a_2g_2').$ We have to show that $(g_1,x,g_2,g_1',g_2',y)\in\Gamma(S')$. For this, it suffices to localize on $S'$. Note that $(g_1,x,g_2)\sim (g_1',y,g_2')$ remains true while passing to étale neighborhoods. Since $\Gamma$ is locally of finite presentation over $S$ (by the definition of $\Gamma$), by a limit argument, we may further assume that $S'$ is the spectrum of a strictly henselian local ring. Then consider the strict localization $\widetilde{S}$ of $S$ at the image of the geometric closed point of $S'$. By \cite[proposition~18.8.8 (iii), (iv)]{EGAIV4}, the morphism $\widetilde{S}\longrightarrow S$ is flat. Therefore, by \cite[2.5, Proposition~6]{BLR}, the formation of $\mathcal{U}$ commutes with the base change to $\widetilde{S}$, hence, we may assume that $S$ is also strictly henselian and the morphism $S'\rightarrow S$ is local. 

\begin{claim}\label{claimquasisection}
    There exist a flat and finitely presented morphism $S_0\longrightarrow S$ and a section $(h_1,h_2)\in( G\times G)(S_0)$ such that $(h_1g_1,x,h_2g_2),\;(h_1g_1',y,h_2g_2')\in \mathcal{V}(S'\times S_0)$, where the open subscheme $\mathcal{V}\subset \Dom(\pi)$ is as in \Cref{strictness} and we implicitly pullback $(h_1,h_2)$ to $(G\times G)(S'\times_S S_0)$. 
\end{claim}

\begin{proof}[Proof of \Cref{claimquasisection}]
    By \Cref{strictness}, we are left with showing that there is a flat and finitely presented morphism $S_0\longrightarrow S$ and a $S$-morphism $S_0 \longrightarrow G\times G$ whose base change to $S'$ lands in the $S'$-dense open subcheme 
$$\mathcal{W}'\coloneq(\mathcal{V}_{x}\cdot(g_1^{-1}, g_2^{-1}))\bigcap \mathcal{V}_{y}\cdot(g_1'^{-1}, g_2'^{-1}))\bigcap (( G\times G)\times S')\subset (G\times G)\times S',$$
where $\mathcal{V}_{x}\subset ( G\times G)\times S'$ (resp., $\mathcal{V}_{y}\subset ( G\times G)\times S'$) is the base change of the open subscheme $\mathcal{V}\subset G\times\overline{\Omega}\times G$ along $x\in\overline{\Omega}(S')$ (resp., $y\in\overline{\Omega}(S')$). To do this, we adopt the strategy of the proof of \cite[corollaire~17.16.2]{EGAIV4}.

Let $s'$ (resp., $s$) be the closed point of $S'$ (resp., $S$). Since $\mathcal{W}'$ is $S'$-dense, we can assume that the special fiber $\mathcal{W}'_{s'}$ contains an open subscheme of the form $\mathcal{W}\times_{k(s)}k(s')$ where $\mathcal{W}$ is a nonempty open subscheme of $(G\times G)_s$ (this follows from a limit argument and \cite[01JR, 01UA]{stacks-project}). We choose a closed point $h\in \mathcal{W}\subset G\times G$. We can choose a system of parameters $(\widetilde{t_i})_{1\leq i\leq n}$ in the local ring $\mathcal{O}_{(G\times G)_s,h}$, because the special fiber $(G\times G)_s$ is a regular scheme. Then we can find an affine open neighborhood $V$ containing $h$ and $n$ sections $(t_i)_{1\leq i\leq n}\subset \Gamma (V,\mathcal{O}_{G\times G})$ lifting $(\widetilde{t_i})_{1\leq i\leq n}$. Let $V'\subset V$ be the closed subscheme cut out by $(t_i)_{1\leq i\leq n}$. By the local criterion for flatness (see \cite[théorème~11.3.8 b') and c)]{EGAIV3}), after shrinking $V'$ if needed, we can assume that $V'$ is flat over $S$. Since $(\widetilde{t_i})_{1\leq i\leq n}$ is a system of parameters,  $\mathcal{O}_{V'_s, h}$ is artinian. Since $h$ is a closed point in $V'_s$, it is isolated in $V'_s$. Hence $\Spec(\mathcal{O}_{V',h})$ is quasi-finite over $S$. As $S$ is henselian, by the Zariski main theorem, the $S$-scheme $\Spec(\mathcal{O}_{V',h})$ is even finite, see \cite[théorème~18.5.11 c')]{EGAIV4}. We take $\Spec(\mathcal{O}_{V',h})$ as $S_0$ and consider the natural $S$-morphism $\epsilon\colon\Spec(\mathcal{O}_{V',h})\rightarrow G\times G$.

Now we show that the $S'$-morphism $\epsilon\times \Id_{S'}\colon S_0\times S'\rightarrow (G\times G)\times S'$ has image in $\mathcal{W'}$. Since $\epsilon$ sends the closed point of $\Spec(\mathcal{O}_{V',h})$ to $h\in \mathcal{W}$ and, by \cite[proposition~2.4.4]{EGA1}, $\epsilon$ preserves generizations (in the sense of \cite[0061]{stacks-project}) of the closed point of $\Spec(\mathcal{O}_{V',h})$, the image of the closed fiber $\epsilon_s$ lies in $\mathcal{W}$. 
Hence, the special fiber $(\epsilon\times \Id_{S'})_{s'}$ has image in $\mathcal{W}\times_{k(s)}k(s')\subset \mathcal{W}'$. Now since $S'$ is henselian, by \cite[proposition~18.5.9 (ii)]{EGAIV4}, the $S'$-finite scheme $\Spec(S_0\times S')$ decomposes as the disjoint union of some local schemes. We already know that the closed points of these local schemes are mapped into $\mathcal{W}'$ by $\epsilon\times \Id_{S'}$. 
Hence, by \cite[proposition~2.4.4]{EGA1}, the whole space $\Spec(S_0\times S')$ also lands in $\mathcal{W}'$, as desired. 
\end{proof}

Now consider the $S_0$-rational morphism 
\begin{flalign*}        \varphi\colon( G\times\overline{\Omega}\times G)\times( G\times G)&\dashrightarrow \overline{\Omega},\\         
     (b_1, \omega ,b_2,b_1',b_2')&\longmapsto \pi((h_1b_1')^{-1},\pi
(h_1b_1,\omega,h_2b_2),(h_2b_2')^{-1}).
\end{flalign*}
By \Cref{associativity}, the two $S_0$-rational morphisms $\varphi$ and $\phi_{S_0}$ equal, where $\phi$ is defined in \Cref{graphlemma}. By \Cref{choiceofequivalence}, the equivalence $(g_1,x,g_2)\sim (g_1',y,g_2')$ implies that $$\pi(h_1g_1,x,h_2g_2)=\pi(h_1g_1',y,h_2g_2').$$
This means that  $\pi((h_1g_1')^{-1},\pi(h_1g_1,x,h_2g_2),(h_2g_2')^{-1})$ is well defined and equals $y$.
By the choice of $\mathcal{V}$ (\Cref{strictness}),  we have that $$(h_1g_1,x,h_2g_2)\in \mathcal{V}(S’\times S_0)\subset \Dom(\pi)(S’\times S_0),$$
$$((h_1g_1’)^{-1}, \pi(h_1g_1,x, h_2g_2), (h_2g_2’)^{-1})\in \Dom(\pi)(S'\times S_0),$$ so $(g_1,x,g_2,g_1’,g_2’)\in \Dom(\varphi)(S') $ because in each step where we apply $\pi$ (to form $\varphi$), the objects are in $\Dom(\pi)$.
Therefore, since $\mathcal{U}$ is the definition domain of $\phi$, by \cite[exposé~XVIII, proposition~1.6]{SGA3II}, it follows that $(g_1,x,g_2, g_1',g_2')\in \mathcal{U}(S')$ and also $(g_1,x,g_2, g_1',g_2',y)$ lies in $\Gamma(S')$.

\end{proof}

\begin{corollary}
    The sheaf $\mathcal{X}$ is an $S$-algebraic space.
\end{corollary}

\begin{proof}
     By a result due to Artin which says that the quotient of an algebraic space with respect to an fppf relation is again an algebraic space (see \cite[corollaire~(10.4)]{Champsalgebrique} or \cite[04S6]{stacks-project}), we are reduced to showing that the two $S$-morphisms $$\pr_{123},\pr_{456}:\Gamma\rightarrow G\times\overline{\Omega}\times G$$ 
     are flat and locally of finite presentation, where again we implicitly identify the target of the projection $\pr_{456}$ with $ G\times\overline{\Omega}\times G$ by switching the last two coordinates.
     
     Observe that $\Gamma$ is isomorphic to $\pr_{12345}(\Gamma)=\mathcal{U}$ which is, by definition, an open subscheme of  $(G\times\overline{\Omega}\times G)\times( G\times G)$. Since $G$ is flat and locally of finite presentation over $S$, by base change, hence so is the morphism $\pr_{123}\vert_{\mathcal{U}}$. To show that $\pr_{456}$ is so, it suffices to note that $\Gamma$ is symmetric under the permutation $(g_1,x,g_2,g_1',g_2',y)\mapsto(g_1',y,g_2',g_1,g_2,x)$ which follows from \Cref{algebraicspace} and that the relation $\sim$ is symmetric \Cref{equivalencerelation}.
\end{proof}

Our next goal is to show that $\mathcal{X}$ is, in fact, a scheme. To do this, we first endow $\mathcal{X}$ with a group action of $G\times G$. 

\begin{definition-proposition}\label{definitionofgrpaction}
    For any $S$-scheme $S'$, consider $a=(a_1,a_2)\in (G\times G)(S')$ and an $x\in \mathcal{X}(S')$ represented by a section
$(g_1,y,g_2)\in( G\times\overline{\Omega}\times G)(S'')$, where $S''$ is an fppf cover of $S'$. We define $ax$ to be the section of $\mathcal{X}(S')$ represented by the section 
$$(a_1g_1, y, a_2g_2)\in ( G\times\overline{\Omega}\times G)(S'').$$
We denote by 
    $$\Theta\colon G\times \mathcal{X}\times G\longrightarrow \mathcal{X}$$
    the resulting action of $G\times G$ on $\mathcal{X}$.
\end{definition-proposition}

\begin{proof}
   We need to check that the action of $G\times G$ on $\mathcal{X}$ is well defined.
   Fix a test $S$-scheme $S'$.
   Let $g=(a_1,a_2)\in (G\times G)(S')$. 
   Let $(g_1, y, g_2)$ and $(g_1',y',g_2')$ be two sections of $G\times\overline{\Omega}\times G$ valued in two different fppf covers $S_1''$ and $S_2''$ of $S'$. By pulling back the two sections to $S_1''\times_{S'} S_2''$, we can assume that $S_1''=S_2''$, which we rename to $S''$. By \Cref{strictness} and \cite[exposé~XVIII, propostion~1.7]{SGA3II}, we can find a section $(b_1, b_2)$ of $G\times G$ valued in some fppf cover of $S''$ such that $\pi(b_1a_1g_1, y, b_2a_2g_2)$ and $\pi(b_1a_1g_1',y',b_2a_2g_2')$ are both well defined. Then we conclude by \Cref{choiceofequivalence}.
\end{proof}

The following lemma is inspired by the fact that the group scheme which is associated to a birational group law on a scheme $X$ contains $X$ as an open schematically dense subscheme (\cite[exposé~XVIII, théorème~3.7 (ii)]{SGA3II}). 
The proof can be done via a direct verification of the definition of being an open immersion, which is, in spirit, similar to \cite[Lemma~3.13]{groupschemeoutofbirationallaw}. Here, we give a simplified proof.

\begin{theorem}\label{representedopenimmer}
    The morphism $\mathbf{j}\colon \overline{\Omega}\rightarrow \mathcal{X}$ which takes a section $a\in\overline{\Omega}(S')$ to the equivalence class represented by $(e,a,e)$ in $\mathcal{X}(S')$, is represented by an open immersion.
\end{theorem}

\begin{proof}
    We first verify that $\mathbf{j}$ is a monomorphism.
    For an $S$-scheme $S'$, let $a_1,a_2\in\overline{\Omega}(S')$ be such that $\mathbf{j}(a_1)=\mathbf{j}(a_2)$. Then there exist an fppf cover $S''$ of $S'$ and a section $(g_1,g_2)\in( G\times G)(S'')$ such that $\pi(g_1,a_1,g_2)=\pi(g_1,a_2,g_2)$. Then by \Cref{choiceofequivalence}, we get $a_1=a_2$ in $\overline{\Omega}(S'')$, hence $a_1=a_2$ in $\overline{\Omega}(S')$.

    To show that $\mathbf{j}$ is an open immersion, we claim that $\mathbf{j}$ is formally smooth. Granted the claim, since by \Cref{algebraicspace}, $\mathcal{X}$ is locally of finite type over $S$, and so is $\overline{\Omega}$, the monomorphism $\mathbf{j}$ is locally of finite presentation by \cite[06Q6]{stacks-project}. Hence, by \cite[0DP0]{stacks-project}, $\mathbf{j}$ is smooth. Futhermore, by \cite[0B8A]{stacks-project} (due to David~Rydh), $\mathbf{j}$ is representable by schemes, and so an open immersion by \cite[025G]{stacks-project}.

    Now we check the formal smoothness of $\mathbf{j}$. Given an infinitesimal thickening $H\hookrightarrow H'$ of $S$-schemes and the following commutative diagram 
    $$\xymatrix{
H \ar@{^{(}->}[d] \ar[r]^a &\overline{\Omega}\ar[d]^{\mathbf{j}}\\
H' \ar[r]_{(g_1,x,g_2)}   \ar@{.>}[ur]^{?}       &\mathcal{X},}$$
where $(g_1,x,g_2)\in(G\times \overline{\Omega}\times G)(H')$,
we need to find a section in $\overline{\Omega}(H')$ fitting into the commutative diagram. Since $\mathbf{j}$ is a monomorphism, by working étale locally and a limit argument, we can assume that $H'$ is the spectrum of a strict henselian local ring. Let $\mathcal{V}\subset \Dom(\pi_{H})$ be an open subscheme obtained by applying \Cref{strictness} to $\pi_{H}$, and let 
$$\mathcal{M}\coloneq (\Dom(\pi_{H'})_x(g_1^{-1},g_2^{-1} ))_H\bigcap \mathcal{V}_a\subset (G\times G)_{H}$$
which is an $H$-dense open subscheme by \Cref{theoremrationalaction}. Since $H\hookrightarrow H'$ is an infinitesimal thickening, $\mathcal{M}$ is $H'$-dense. Now \cite[\S~2.3, Proposition~5]{BLR} gives rise to a section $(c_1,c_2)\in\mathcal{M}(H')$. By the choice of $\mathcal{M}$ and the commutativity of the square, $c\coloneq \phi(c_1g_1,x,c_2g_2, c_1,c_2)\in \overline{\Omega}(H')$ is well-defined, where $\phi$ is introduced in \Cref{graphlemma}. By the definition of $(c_1,c_2)$, we have that
$\pi(c_1g_1,x,c_2g_2)=\pi(c_1,c,c_2),$
and by \Cref{choiceofequivalence}, we have $\pi(c_1,c,c_2)=\pi(c_1,a,c_2)$, as desired.
\end{proof}

\begin{proposition}\label{schematicnature}
    The algebraic space $\mathcal{X}$ is a separated, smooth, finitely presented scheme over $S$.
\end{proposition}

\begin{proof}
    By the functoriality of the construction of $\mathcal{X}$ (\Cref{defintionofcompactification}) and base change, we may assume that $S=\Spec(\mathbb{Z})$.
    By \Cref{representedopenimmer}, the $\overline{\Omega}$ is an open subspace of $\mathcal{X}$. By the definition of the action of $G\times G$ on $\mathcal{X}$ (\Cref{definitionofgrpaction}), we have $\mathcal{X}=(G\times G)\cdot\overline{\Omega} $. Then the proposition follows from \cite[Section~6.6, Theorem~2 (a) and (b)]{BLR} except for the smoothness of $\mathcal{X}$ over $S$. 
    
    To show the \emph{smoothness} of $\mathcal{X}$ over $S$, since $\mathcal{X}$ is locally of finite presentation over $S$, we can work étale locally over $S$ and use a limit argument to assume that $S$ is the spectrum of a strictly henselian local ring.
    In this case, by the smoothness of $\overline{\Omega}$ over $S$, we only need to show that $\mathcal{X}$ is covered by the translates of $\overline{\Omega}$ by the sections in $(G\times_S G)(S)$. 
    
    For this, let $S'$ be a test $S$-scheme, and let $\overline{x}\in \mathcal{X}(S')$ represented by $(g_1,\omega,g_2)\in (G\times_{S} \overline{\Omega}\times_S G)(S'')$ where $S''$ is an fppf-cover of $S'$. Without loss of generality, we assume that $S''=S'$. It suffices to find an open covering $S'=\bigcup_{i\in I} S'_{i}$ such that, for each open $S'_i$, there exists a section $(c_1^i,c_2^i)\in(G\times_S G)(S)$ such that $\omega'\coloneq \phi(g_1\vert_{S'_{i}},\omega\vert_{S'_{i}},g_2\vert_{S'_{i}},c_1^i,c_2^i)$ is well-defined, where $\phi$ is defined in \Cref{graphlemma}. If so, by \Cref{algebraicspace}, we have that $(c_1^i,\omega',c_2^i)\sim (g_1\vert_{S'_{i}},\omega\vert_{S'_{i}},g_2\vert_{S'_{i}})$ which means that $\overline{x}\vert_{S_i}$ lies in $(c_1^i,c_2^i)\cdot \overline{\Omega}$, as desired. 

    Without loss of generality, we assume that $S''=S'$. Since the problem is local over $S'$, by working Zariski locally on $S'$ and a limit argument, we can assume that $S'$ is a local scheme. In this case, by the definition of $\phi$ and \Cref{theoremrationalaction}, the condition on the sought $(c_1,c_2)$ amounts to requiring that $(c_1,c_2)(S')\subset U$, where $U\subset (G\times_S G)\times_{S} S'$ is an open $S'$-dense subscheme. Since now $S'$ is a local scheme, it suffices to require that the closed point of $S'$ be sent into $U$ by $(c_1,c_2)$. Then we are reduced to the case where $S'=\Spec(k)$ with $k$ a field. Thus, the existence of such a section $(c_1,c_2)$ is ensured by Artin's \cite[\S~5.3, Lemma~7]{BLR} which asserts that, for any point $t\in S$, the set $\{a(t)\vert a\in (G\times_S G)(S)\}$ is dense in the fiber $(G\times_S G)_t$.
\end{proof}

\subsection{Relation with wonderful compactifications}\label{relationwithwonder}
We are going to show that the geometric fibers of $\mathcal{X}$ are classical wonderful compactifications (which we reviewed in \Cref{reviewofwonderful}). For this, recall that the wonderful compactification $\mathbf{X}_s$ of $G_{\overline{k(s)}}$ contains an open subscheme called the big cell which is canonically isomorphic to $(\overline{\Omega})_{\overline{k(s)}}$, see, for instance, \cite[Proposition~6.1.7]{BrionKumar}.

\begin{proposition}\label{groupembedtocompac}
    The open embedding of \Cref{Omegaembedformula} extends to 
    a $(G\times G)$-equivariant open immersion $\mathbf{i}\colon G\hookrightarrow\mathcal{X}$    
    such that every geometric fiber of $\mathcal{X}$ over a point $s\in S$ is identified with the wonderful compactification of $G_{\overline{k(s)}}$. 
\end{proposition}
    
\begin{proof}

For a point $s\in S$, by the functoriality of the definition of $\mathcal{X}$ (\Cref{defintionofcompactification}), the geometric fiber $\mathcal{X}_{\overline{k(s)}}$ of $\mathcal{X}$ is the quotient sheaf $G_{\overline{k(s)}}\times(\overline{\Omega})_{\overline{k(s)}}\times G_{\overline{k(s)}}/\sim_s$, where $\sim_s$ stands for the equivalence relation obtained by applying \Cref{relationonOmega} to $G_{\overline{k(s)}}$. To define a morphism from $\mathcal{X}_{\overline{k(s)}}$ to $\mathbf{X}_s$, by the universal property of sheafication, it suffices to define a morphism $\chi_s^+$ at the level of presheaves. For a $\overline{k(s)}$-scheme $K$ and a section $(g_1,x,g_2)\in(G_{\overline{k(s)}}\times(\overline{\Omega})_{\overline{k(s)}}\times G_{\overline{k(s)}})(K)$, we define the image of $(g_1,x,g_2)$ under $\chi_s^+$ to be $g_1xg_2^{-1}$ in $\mathbf{X}_s(K)$. The fact that $\chi_s^+$ is well-defined follows from the fibral description of $\pi$ in \Cref{theoremrationalaction}. The associated morphism from $\mathcal{X}_{\overline{k(s)}}$ to $\mathbf{X}_s$ is denoted by $\chi_s$. Note that, by \Cref{definitionofgrpaction}, this $\chi_s$ is $(G_{\overline{k(s)}}\times G_{\overline{k(s)}})$-equivariant. By \Cref{representedopenimmer} and the definition of $\chi_s$, the restriction $\chi_s\vert_{(\overline{\Omega})_{\overline{k(s)}}}$ is an open immersion onto the big cell of $\mathbf{X}_s$. Since $\chi_s$ is a monomorphism, it is then an open immersion. By \Cref{translationbigcell}, $\chi_s$ is an isomorphism.

For a section $g\in G(S')$ where $S'$ is an $S$-scheme, we define $\mathbf{i}\colon G\rightarrow\mathcal{X}$ by taking $\mathbf{i}(g)$ to the equivalence class of $(g,\nu(e),e)$ in $\mathcal{X}(S')$, where we recall that $e\in G(S)$ is the identity section. The fact that $\mathbf{i}$ is $(G\times G)$-equivariant follows from  \Cref{theoremrationalaction} (1) and \Cref{relationonOmega}.  Moreover, by our construction, it is clear that the fiber $\mathbf{i}_s$ is transferred into the natural open immersion $G_{\overline{k(s)}}\hookrightarrow \mathbf{X}_s$ by $\chi_s$.
    
    \begin{claim}\label{intersectionofbigcell}
        We claim that $G\bigcap \overline{\Omega}=\Omega_G$, here the intersection indicates the pullback of $\overline{\Omega}$, which is an open subscheme of $\mathcal{X}$ (\Cref{representedopenimmer}), along $\mathbf{i}$.
    \end{claim}

   \begin{proof}[Proof of \Cref{intersectionofbigcell}] 
       We choose a section $g\in\Omega_G(S')$. By \Cref{definitionofgrpaction}, $\mathbf{i}(g)$ is the section of $\mathcal{X}$ represented by the section $(g,\nu(e),e)\in( G\times\overline{\Omega}\times G)(S')$. Note that, by \Cref{relationonOmega}, $(g,\nu(e),e)\sim(e, \nu(g),e)$, where $\nu$ is defined in \Cref{Omegaembedformula}. Hence $\Omega_G\subset G\bigcap \overline{\Omega}$.
    To see that the inclusion is an identity, for each point $s\in S$, inside the geometric fiber $\mathcal{X}_{\overline{k(s)}}$, by \cite[Proposition~6.1.7]{BrionKumar}, we have that $(\Omega_G)_{\overline{k(s)}}=G_{\overline{k(s)}}\bigcap \overline{\Omega}_{\overline{k(s)}}$. Therefore, it follows that $\Omega_G= G\bigcap \overline{\Omega}$.
   \end{proof}
    
    Now as $\Omega_G$ is an open subscheme of $G$ \cite[exposé~XXII, proposition~4.1.2]{SGA3III} and $\mathcal{X}$ is covered by the translates of $\overline{\Omega}$ under the action of $G\times G$, it follows that $\mathbf{i}$ is an open immersion.
    
\end{proof}

\begin{remark}
    In the proof of \Cref{groupembedtocompac}, we construct $\mathbf{i}$ via the left action of $G$ on $\mathcal{X}$. In the view of the condition (1) of \Cref{theoremrationalaction}, it is equivalent to use the right action of $G$.
\end{remark}

\begin{remark}
Classical wonderful compactifications over a field can be realized as the closures of $G$ via the embedding into the projectivization of the endomorphism space of an algebraic representation of $G$ (\Cref{reviewofwonderful}) or via the embedding into certain Grassmannian of Lie algebra \cite[\S~6]{completesymmetricvarieties} or via the embedding into certain Hilbert scheme of flag variety \cite{compactificationHilbertsch}. All these three embeddings can be defined over general base scheme. The geometric fibers of the closure of $G$ inside an ambient projective space \emph{a priori} can not be identified with the classical wonderful compactifications because in general schematic closure does not commute with non flat base change. However, it is reasonable to expect that, over general base, the closure as before \emph{a posteriori} does commute with non flat base change by showing that the closure is isomorphic to our compactification $\mathcal{X}$.
\end{remark}

\subsection{Projectivity of the compactification}
In \Cref{relationwithwonder}, we saw that the geometric fibers of $\mathcal{X}$ over $S$ are projective. However, this does not imply the projectivity of $\mathcal{X}$ because the projectivity does not satisfy even Zariski descent, see \cite[Chapter~II, Exercise~7.13]{HartshorneGTM52} for a counterexample. In this section, we will show that $\mathcal{X}$ is projective over $S$.
We first obtain the properness of $\mathcal{X}$.

\begin{theorem}\label{properness}
    The scheme $\mathcal{X}$ is proper over $S$.
\end{theorem}

\begin{proof}
    \emph{First method}: we are going to verify the valuative criterion for properness \cite[théorème~7.3.8]{EGA2}. Since, by \Cref{schematicnature}, $\mathcal{X}$ is separated over $S$, and using \cite[remarques~7.3.9]{EGA2}, we can assume that the base scheme $S$ is the spectrum of a complete discrete valuation ring $R$. Let $K$ be the fraction field of $R$, and let $\epsilon\in R$ be a uniformizer.
    
    Since $\mathcal{X}$ is separated over $S$, using \cite[Lemma~4.1.1]{valuationcriopendense}, it suffices to show that, for any $x\in G(K)$, there exists a section $y\in \mathcal{X}(R)$, such that the following diagram
     $$\xymatrix{
\Spec (K) \ar[r]^-{x}\ar[d] & G\ar[d]^{\mathbf{i}}\\
\Spec (R) \ar[r]^-{y} & \mathcal{X}
}$$ 
commutes, where $\mathbf{i}$ is the open immersion from \Cref{groupembedtocompac}. By the Iwahori--Cartan decomposition (see, \cite[Remark~3.5, Theorem~4.1]{Iwahoricdecomp}), we have 
$$G(K)=G(R)T(K)G(R)=\bigcup_{\lambda\in X_{\ast}(T)}G(R)\lambda(K)G(R),$$
where $X_{\ast}(T)$ is the cocharacter lattice of $T$. Since the Weyl group acts transitively (and freely) on the Weyl chambers (\cite[exposé~XXI, corollaire~5.5]{SGA3III}), it is enough to show that the section $\lambda(\epsilon^{-1})$ with $\lambda$ a dominant cocharacter can be lifted to an $R$-section of $\overline{T}$. This follows from \Cref{Omegaembedformula}.

\emph{Second method}: since $\mathcal{X}$ is separated over $S$ and of finite presentation over $S$ (\Cref{schematicnature}) and $\mathcal{X}(S)\neq \emptyset$ (\Cref{groupembedtocompac}), by \cite[corollaire~15.7.11]{EGAIV3}, the properness of $\mathcal{X}$ over $S$ follows from the properness of the geometric fibers of $\mathcal{X}$ (\Cref{groupembedtocompac}).
\end{proof}

\begin{remark}\label{remarkproperness}
    Compared with the second proof which is extrinsic, the first proof of \Cref{properness} is intrinsic in the sense that it does not need any embedding of $G$ to obtain the properness. Even in the case when the base scheme is a field, the first proof can serve as an equivalent but intrinsic reason for the properness of classical wonderful compactification. We expect that the first method can help when we explore relative version of more general wonderful varieties (in the sense of Luna \cite{Lunawonderfulvarieties}) whose geometries are less explored in positive characteristic.
\end{remark}

Now we can strengthen \Cref{properness} to get the projectivity of $\mathcal{X}$ over $S$. For this, the crux is to construct an $S$-ample line bundle over $\mathcal{X}$. The existence of such a line bundle follows from a general fact due to M. Raynaud on which we elaborate in the following corollary. See also \Cref{descentampleCartier} (and its proof) for another way to construct $S$-ample line bundles over $\mathcal{X}$ (assuming \Cref{projectivity}).

\begin{corollary}\label{projectivity}
    The scheme $\mathcal{X}$ is projective over $S$.
\end{corollary}

\begin{proof}   
     We can utilize
    \cite[Section~6.6, Theorem~2 (d)]{BLR} to get the quasi-projectivity of $\mathcal{X}$ over $S$. For completeness, we include a detailed proof here. More precisely, first note that the scheme $\mathcal{X}$ is the base change of a scheme $\mathcal{X}_0$ over $\Spec(\mathbb{Z})$ obtained by applying \Cref{defintionofcompactification} to the Chevalley group that shares the same root datum with $G$. Thus it suffices to show that $\mathcal{X}$ is quasi-projective over $S$ when $S=\Spec(\mathbb{Z})$. In this case, by \cite[corollaire~21.12.7]{EGAIV4}, every irreducible component of $\mathcal{X}\backslash \overline{\Omega}$ is of codimension one. Then by the Ramanujam--Samuel theorem \cite[théorème~21.14.3]{EGAIV4}, any effective Weil divisor $D$ whose support is $\mathcal{X}\backslash \overline{\Omega}$ is an effective Cartier divisor. Now we only need to show that such a divisor $D$ is $S$-ample. To see this, first, 
    since $S$ is now an affine scheme, by \cite[chapitre~II, corollaire~4.6.6]{EGA2}, it suffices to show that the invertible sheaf $\mathcal{L}\coloneq\mathcal{L}(D)$ is ample. Furthermore, by \cite[corollaire~2.7.2]{EGAIV2}, we can work étale locally over $S$, and by a limit argument, we are reduced to the case where $S$ is a strictly henselian discrete valuation ring. 
    
    By \cite[proposition~4.5.6(i)]{EGA2}, we have the flexibility to replace $\mathcal{L}$ by some power of itself. Hence, by \cite[\S~6.3, Theorem~1]{BLR}, we can assume that $\mathcal{L}$ satisfies the theorem of the square, cf., \cite[\S~6.3]{BLR}. Hence, $gD-D$ is linearly equivalent to $-(g^{-1}D-D)$ for any $g\in (G\times G)(S)$. This means that there exists a section $l\in \Gamma(\mathcal{X}, \mathcal{L}^{2})$ such that 
$$\mathcal{X}_l=\mathcal{X}\backslash\Supp(gD+g^{-1}D)=g\overline{\Omega}\bigcap g^{-1}\overline{\Omega},$$
where the second equality follows from our assumption on the support of the divisor $D$. Since, by \Cref{schematicnature}, $\mathcal{X}$ is separated over $S$, the intersection $g\overline{\Omega}\bigcap g^{-1}\overline{\Omega}$ is affine over $S$, see, for instance, \cite[01KP]{stacks-project}. 

Now to show the ampleness of $\mathcal{L}$, it suffices to show that every point $x\in\mathcal{X}$ is contained in an affine open subscheme of the form $g\overline{\Omega}\bigcap g^{-1}\overline{\Omega}$ for some $g\in (G\times_S G)(S)$. To do this, suppose that $x$ lies over a point $s\in S$, let $k'$ (resp., $k$) be the residue field of $x$ (resp., $s$), and we make the following claim:

\begin{claim}\label{claim1}
    There exists an open schematically dense subscheme $Z_s\subset (G\times_S G)_s$ such that, for any $g\in Z_s(k)$, we have that $x\in g\overline{\Omega}_s\bigcap g^{-1}\overline{\Omega}_s.$
\end{claim}

\begin{proof}[Proof of \Cref{claim1}]
To show the claim, we can assume that $x$ is a closed point. Then, by Hilbert's Nullstellensatz (cf. \cite[Theorem~1.7]{GortzWedhorn}), the field extension $k\subset k'$ is finite. Let $W$ be the pullback of $\overline{\Omega}_s\times_k k'$ along the $k'$-morphism
\begin{equation*}
    \begin{split}
        (G\times_S G)_s\times_k k'&\longrightarrow \mathcal{X}_s\times_k k'\\
        g&\longmapsto g\cdot x,
    \end{split}
\end{equation*}
and let $W^{-1}$ be the inverse of $W$ with respect to the group law of $(G\times_S G)_s\times_k k'$. Then the condition that $g\in (W\bigcap W^{-1})(k')$ gives rise to $x\in g(\overline{\Omega}_s\times_k k')\bigcap g^{-1}(\overline{\Omega}_s\times_k k')$. Note that, since $\mathcal{X}=(G\times_S G)\cdot \overline{\Omega}$ and $G$ has connected geometric fibers, $W\bigcap W^{-1}$ is open and schematically dense
in $(G\times_S G)_s\times_k k'$. Therefore, by the finiteness of $k'$ over $k$ and \cite[01JR, 01UA]{stacks-project}, there is an open schematically dense subscheme of $W\bigcap W^{-1}$ which descends to an open schematically dense subscheme $Z_s\subset (G\times_S G)_s$, as desired.
\end{proof}

\noindent Now, if $s$ is the closed point of $S$, by \cite[\S~2.3, Proposition~5]{BLR}, there exists a section $g\in (G\times_S G)(S)$ such that the base change of $g$ lies in $Z_s(k)$. If $s$ is the generic point, then the schematic closure of $(G\times_S G)_s\backslash Z_s$ in $G\times_S G$ is flat over $S$, and hence is nowhere dense in $G\times_S G$. Hence, by applying \emph{loc. cit.} to the complement of the schematic closure in $G\times_S G$, we can also find a section $g\in (G\times_S G)(S)$ such that the base change of $g$ lies in $Z_s(k)$ in this case. This finishes the proof of the ampleness of $\mathcal{L}$.

    Finally, by \cite[Lemma~0BCL~(1)]{stacks-project}, the corollary follows from \Cref{properness}.
\end{proof}

\begin{remark}
    Our method of constructing the projective scheme $\mathcal{X}$ containing $G$ as an open subscheme fails for non adjoint reductive groups. A key feature of our approach is that for the adjoint group $G$, the set of simple roots forms a basis of the character space $X_{\mathbb{R}}$ and cuts out a rational cone $C_0$ (the negative Weyl chamber) in the cocharacter space $X^{\vee}_{\mathbb{R}}$ such that the translates of $C_0$ under the Weyl group form a complete fan in $X^{\vee}_{\mathbb{R}}$. This feature is used to get the properness of $\mathcal{X}$, see the proof of \Cref{properness} and \Cref{remarkproperness}. In the non adjoint case, there is no such a basis in the character space. In fact, since the action of the Weyl group on the cocharacter space is generated by simple reflections, the cone $C_0$ cut out by such a basis must lie in some Weyl chamber. Since the fan of the translates of $C_0$ is complete, the cone $C_0$ must be a Weyl chamber. Hence this basis is the image of the set of simple roots by a Weyl group element. This implies the root datum of $G$ is adjoint.\footnote{We thank Will Sawin for pointing this out to us.}
\end{remark}

\begin{remark}
    A wonderful compactification (which is not necessarily smooth) $\overline{G}$ for a general split reductive group $G$ over a field has been proposed in \cite[2.4]{secondadjointness}, and has been applied to give a geometric proof of the second adjointness of Bernstein for p-adic groups, see loc. cit.
\end{remark}

\begin{remark}   
    It is also possible to replace $G\times\overline{\Omega}\times G$ by $G\times\overline{T}\times G$ in the construction of $\mathcal{X}$, which is similar to \cite{gabbercompactreport} where Gabber sketched the construction of a proper compactification for pseudo-reductive groups of minimal type over fields. Another possible model for $\mathcal{X}$, which is similar to \cite[Chapter~5]{BLR}, is to realize it as the subsheaf of $\mathcal{R}(\overline{\Omega})$ which is the sheaf of rational endomorphisms of $\overline{\Omega}$. 
\end{remark}

\section{Divisors}\label{sectiondivisors}
Divisors of a scheme often reflect many geometric properties of the scheme itself, for instance, in Luna's program for the classification of wonderful varieties (in characteristic zero), a kind of divisors which are called ``colors'' plays an important role, see \cite{Lunavarietespherique} for details. In this section, we will study the divisors of the scheme $\mathcal{X}$ constructed in \Cref{sectioncompactification} and their interactions with the $(G\times G)$-orbits of $\mathcal{X}$. In this section, we shall keep notations of \Cref{sectioncompactification}. 

We denote by $\underline{G}$ the unique Chevalley group over $\Spec(\mathbb{Z})$ which has the same root datum as $G$, and denote by $\underline{\mathcal{X}}$ the scheme obtained by applying \Cref{defintionofcompactification} to $\underline{G}$.
The following result has been established for classical wonderful compactifications, see, for instance, \cite[Chapter~6, Lemma~6.1.9]{BrionKumar}, whose argument can be carried over to our situation with mild modifications, except that it is more natural to consider relative Picard schemes instead of absolute Picard groups.

We let $G_{\text{sc}}\rightarrow G$ be the simply connected cover of $G$, and let $B_{\text{sc}}\subset G_{\text{sc}}$ (resp., $T_{\text{sc}}\subset G_{\text{sc}}$) be the preimage of $B$ (resp., $T$). Let $\{\omega_i\}\subset X^*(T_{\text{sc}})$ be the set of fundamental weights defined by $B_{\text{sc}}$ so that $\langle w_i, \alpha_j^\vee\rangle =\delta_i^j$. Let $w_0\in W$ be the longest element in the Weyl group.

\begin{theorem}\label{boundarydivisor}
    The complement $\mathcal{X}\backslash\overline{\Omega}$ is the union of the relative effective Cartier divisors $\overline{B\dot{s_i}B^-}$ (See \cite[062T]{stacks-project}), where $\dot{s_i}\in\Norm_G(T)(S)$ is a lifting of $s_i\in W(S)$ which is the simple reflection defined by the simple root $\alpha_i\in \Delta$ (\cite[exposé~XXII, corollaire~3.8]{SGA3III}) and bar indicates schematic closure in $\mathcal{X}$. Moreover, the relative Picard functor $P_{\mathcal{X}/S}$, which is the fppf sheaf associated to the functor
    $$ (\text{Sch}/S)^{\text{op}}\longrightarrow \text{Grp},\;\;T\longmapsto \Pic(\mathcal{X}\times_S T)$$
    from the category of $S$-schemes to the category of groups, 
    is represented by the constant group scheme $\underline{\mathbb{Z}^{\oplus l}}_S$ with basis the classes of the 
$\overline{B \dot{s_i} B^-}$.
\end{theorem}

Before coming to the proof of \Cref{boundarydivisor}, as a preparation, we need the following lemmas.

\begin{lemma}
    Assume that $S=\Spec(\mathbb{Z})$. The schematic closure $\overline{B\cdot \dot{s_i} \cdot B^-}\subset \mathcal{X}$ is an effective Cartier divisor.
\end{lemma}

\begin{proof}
   As $\mathcal{X}$ is separated by \Cref{schematicnature}, the affineness of $\overline{\Omega}$ ensures that the open embedding $\mathbf{j}\colon\overline{\Omega}\hookrightarrow\mathcal{X}$ in \Cref{representedopenimmer} is affine. By \cite[corollaire~21.12.7]{EGAIV4}, the complement $\mathcal{X}\backslash \overline{\Omega}$ is of pure codimension 1.
   We claim that $G\backslash \overline{\Omega}$ is schematically dense in $\mathcal{X}\backslash \overline{\Omega}$, where $G\backslash \overline{\Omega}\subset G$ and $\mathcal{X}\backslash \overline{\Omega} \subset \mathcal{X}$ are endowed with the reduced closed subscheme structures. Since $\mathcal{X}\backslash \overline{\Omega}$ is reduced, by \cite[Remark~9.20 (1)]{GortzWedhorn}, it suffices to show that $G\backslash \overline{\Omega}$ is topologically dense in $\mathcal{X}\backslash \overline{\Omega}$. Hence, by base change to geometric fibers and \Cref{groupembedtocompac}, the claim follows from the same result for classical wonderful compactifications \cite[Chapter~6, proof of Lemma~6.1.9]{BrionKumar}. 
   
   Furthermore, by the Bruhat decomposition \cite[exposé~XXII, théorème~5.7.4~(i)]{SGA3III}, combined with the fact that $G\bigcap \overline{\Omega}=\Omega_G$ (\Cref{intersectionofbigcell}), we have
   $$G\backslash\overline{\Omega}=\coprod_{w\in W(S)\backslash 1} B\cdot \dot{w} \cdot B^-=\bigcup_{s_i}\overline{B\cdot \dot{s_i} \cdot B^-}$$
   holds for the underlying topological space, where $\dot{w}\in \Norm_G(T)(S)$ is a lifting of $w$, the bar indicates schematic closure in $G\backslash \overline{\Omega}$, and the second equality follows from the property of the Bruhat order \cite[8.5.4, Proposition~8.5.5]{linearalgrpSpringer}. Therefore, after taking schematic closure in $\mathcal{X}$, we obtain that the equality $$\mathcal{X}\backslash\overline{\Omega}=\bigcup_{s_i }\overline{B\cdot \dot{s_i} \cdot B^-}$$
   holds for the underlying topological spaces, and $\overline{B\cdot \dot{s_i} \cdot B^-}$ is reduced by \cite[Remark~10.32]{GortzWedhorn}.

   As now $S$ is a Dedekind scheme, by \cite[Chapter~III, Proposition~9.7]{HartshorneGTM52}, these boundary divisors $\overline{B\cdot \dot{s_i} \cdot B^-}$ are flat over $S$ because $\overline{B\cdot \dot{s_i} \cdot B^-}$ dominates $S$. As $\mathcal{X}$ is locally factorial ($S=\Spec(\mathbb{Z})$), by \cite[Chapter~II, Proposition~6.11]{HartshorneGTM52}, $\overline{B\cdot \dot{s_i} \cdot B^-}$ is an effective Cartier divisor, hence is a relative effective Cartier divisor. 
\end{proof}

\begin{lemma}\label{fiberbasechangelemma}
    Assume that $S=\Spec(\mathbb{Z})$. The equality $ (\overline{B\cdot \dot{s_i} \cdot B^-})_{\overline{k(s)}}=\overline{B_{\overline{k(s)}}\cdot \dot{s_i} \cdot B_{\overline{k(s)}}^-}$ holds as an equality of Cartier divisors, where the latter bar indicates taking schematic closure in $\mathcal{X}_{\overline{k(s)}}$.
\end{lemma}

\begin{proof}

For a point $s\in S$, the geometric fiber $(\overline{B\cdot \dot{s_i} \cdot B^-})_{\overline{k(s)}}$ certainly contains $\overline{B_{\overline{k(s)}}\cdot \dot{s_i} \cdot B_{\overline{k(s)}}^-}$, where the closure is taken in the wonderful compactification of $G_{\overline{k(s)}}$. By \cite[056Q]{stacks-project}, $(\overline{B\cdot \dot{s_i} \cdot B^-})_{\overline{k(s)}}$ is again an effective Cartier divisor.

\begin{claim}\label{claimdivisor}
    The geometric fiber $(\overline{B\cdot \dot{s_i} \cdot B^-})_{\overline{k(s)}}$ cannot contain any other boundary divisor $\overline{B_{\overline{k(s)}}\cdot \dot{s_j} \cdot B_{\overline{k(s)}}^-}$ for $j\neq i$.
\end{claim}

\begin{proof}[Proof of \Cref{claimdivisor}]
    Since the formation of $B\cdot \dot{s_i} \cdot B^{-}$ commutes with passage to geometric fibers \cite[lemme~5.7.3 (i)]{SGA3II}, by the usual Bruhat decomposition \cite[8.5.4]{linearalgrpSpringer}, we have that $B\cdot \dot{s_i} \cdot B^-$ is contained in $ G\backslash \bigcup_{w\in W, ww_0\nleq s_iw_0} B\cdot w\cdot B^-,$ which is a closed subscheme of $G$. By taking closure in $G$ and passing to the geometric fiber, we have 
   $$(\overline{B \cdot \dot{s_i} \cdot B^-})_{\overline{k(s)}}\subset G_{\overline{k(s)}}\backslash \bigcup_{w\in W, ww_0\nleq s_iw_0} (B\cdot w\cdot B^-)_{\overline{k(s)}}=\overline{B_{\overline{k(s)}}\cdot \dot{s_i} \cdot B_{\overline{k(s)}}^-}\subset G_{\overline{k(s)}},$$
   where we use the functorial description of Bruhat cells, see \cite[exposé~XXII, lemme~5.7.3]{SGA3III}
   Now if the claim does not hold, then $\dot{s_j}\in (\overline{B \cdot \dot{s_i} \cdot B^-})_{\overline{k(s)}}$ for some $j\neq i$, and by intersecting with $G_{\overline{k(s)}}$, we have that $\dot{s_j}$ lies in the closure of $B_{\overline{k(s)}}\cdot \dot{s_i} \cdot B_{\overline{k(s)}}^-$ in $G_{\overline{k(s)}}$, which is a contradiction with \cite[8.5.4]{linearalgrpSpringer}.
\end{proof}
Therefore, since we already know that $\overline{B\cdot \dot{s_i} \cdot B^-}$ is a relative Cartier divisor, by \Cref{claimdivisor} and the description of the irreducible components of the complement of the big cell in the classical wonderful compactification  of $G_{\overline{k(s)}}$ (cf. \cite[Lemma~6.1.9]{BrionKumar}), we have that the equality 
\begin{equation}\label{basechangeofboundarydivisor}
       (\overline{B\cdot \dot{s_i} \cdot B^-})_{\overline{k(s)}}=\overline{B_{\overline{k(s)}}\cdot \dot{s_i} \cdot B_{\overline{k(s)}}^-}
   \end{equation}
 holds for the underlying topological spaces.
Now to show that \Cref{basechangeofboundarydivisor} is an equality of Cartier divisors, it suffices to check that $\overline{B\cdot \dot{s_i} \cdot B^-}$ is geometrically reduced, or equivalently that $\overline{B\cdot \dot{s_i}w_0 \cdot B}$ is so. 
Since $T_{\text{sc}}$ is a split tours, by the flat base change to geometric generic fiber and \cite[Proposition~6.1.11 (ii) (iv)]{BrionKumar}, the canonical section $\tau_i\in \Gamma(\mathcal{X},\mathcal{O}_{\mathcal{X}}(\overline{B\cdot \dot{s_i}w_0 \cdot B}))$ is of weight $(\omega_i, -w_0\omega_i)$ under the $(B_{\text{sc}}\times B_{\text{sc}})$-action. If $(\overline{B\cdot \dot{s_i}w_0 \cdot B})_{\overline{k(s)}}$ is non-reduced for some point $s\in S$, we would get a contradiction with the fact that the fundamental weight $\omega_i$ is indivisible in the lattice $X^*(T_{\text{sc}})$.
\end{proof}

\begin{proof}[Proof of \Cref{boundarydivisor}]
   We first assume that $S=\Spec(\mathbb{Z})$.
   Take a Weil divisor $D$ in $\mathcal{X}$. As $\overline{\Omega}$ is an affine space over $\Spec(\mathbb{Z})$, by \cite[Proposition~11.40]{GortzWedhorn}, $D$ is linearly equivalent to a divisor spanned by these boundary divisors. If $D$ is a linear combination of boundary divisors and is equivalent to a principal divisor associated to a rational function $h$ in the function field $K(\mathcal{X})$, after restricting to $\overline{\Omega}$ which is an affine space over $\Spec(\mathbb{Z})$, we have $h\in\mathbb{Z}^{\times}$, hence $D$ itself is zero. Therefore $\text{Pic}(\mathcal{X})$ is freely generated by these boundary divisors.

   Since, by \Cref{projectivity}, $\mathcal{X}$ is smooth and projective over $\Spec(\mathbb{Z})$ and by \Cref{groupembedtocompac} the geometric fibers of $\mathcal{X}/\Spec(\mathbb{Z})$ are reduced and irreducible, according to the result \cite[\S~8.2 Theorem~1]{BLR} of Grothendieck, the relative Picard functor $P_{\mathcal{X}/\Spec(\mathbb{Z})}$ is representable by a seperated $S$-scheme, say $\Pic_{\mathcal{X}/S}$, which is locally of finite presentation over $S$.
   By \Cref{groupembedtocompac}, the generic fiber of $\mathcal{X}$ over $S$ is geometrically integral, hence, by \cite[Proposition~5.51]{GortzWedhorn}, $\mathbb{Q}$ is algebraically closed in the function field $K(\mathcal{X})$. Then by \cite[corollaire~4.3.12]{EGA3I}, the pushforward of the structure sheaf $\mathcal{O}_{\mathcal{X}}$ to $S$ equals to $\mathcal{O}_S$. By \cite[\S~8.1, Proposition~4]{BLR}, $\Pic_{\mathcal{X}/S}(S)$ equals the Picard group $\Pic(\mathcal{X})$, which is, as we have seen, the free abelian group $\mathbb{Z}^{\oplus l}$. Hence we can define a morphism $$\theta\colon\underline{\mathbb{Z}^{\oplus l}}_S\longrightarrow\Pic_{\mathcal{X}/S}$$
   by sending an $l$-tuple of integers $(a_i)_{i=1,..., l}$ to the divisor $\Sigma_{i=1}^l a_i (\overline{B\cdot \dot{s_i} \cdot B^-})$. By applying \cite[Corollary~6.2.8]{BrionKumar} to the wonderful compactification $\mathcal{X}_{\overline{k(s)}}$ for a point $s\in S$, we have $H^1(\mathcal{X}_{\overline{k(s)}},\mathcal{O}_{\mathcal{X}_{\overline{k(s)}}})=0$. Combining this with \cite[\S~8.4 Theorem~1 (b)]{BLR}, it follows that  $(\Pic_{\mathcal{X}/S})_{\overline{k(s)}}$, which is isomorphic to $\Pic_{\mathcal{X}_{\overline{k(s)}}/\overline{k(s)}}$, has dimension zero and is smooth over $\overline{k(s)}$. By \cite[Lemma~6.1.9]{BrionKumar}, the underlying topological space of $(\Pic_{\mathcal{X}/S})_{\overline{k(s)}}$ is the free abelian group $\mathbb{Z}^{{\bigoplus}l}$ with basis the $\overline{B_{\overline{k(s)}}\cdot \dot{s_i} \cdot B_{\overline{k(s)}}^-}$. Hence, by \cite[Corollary~5.8]{GortzWedhorn} and \Cref{basechangeofboundarydivisor}, the restriction of $\theta$ to each geometric fiber is an isomorphism, and then by \cite[corollaire~17.9.5]{EGAIV4}, $\theta$ itself is an isomorphism (here we use the flatness of $\Pic_{\mathcal{X}/\Spec(\mathbb{Z})}$ which follows from \cite[III, Proposition~9.7]{HartshorneGTM52}).

   Now let us go back to the general base $S$. By \Cref{defintionofcompactification}, $\mathcal{X}$ is the base change of $\underline{\mathcal{X}}$ along the natural morphism from $S$ to $\Spec(\mathbb{Z})$. By \Cref{fiberbasechangelemma}, each boundary divisor $\overline{B\cdot \dot{s_i} \cdot B^-}$ descends to the boundary divisor of $\underline{\mathcal{X}}$ which is the schematic closure of the obvious $\mathbb{Z}$-model of $B\cdot \dot{s_i} \cdot B^-$ in $\underline{\mathcal{X}}$, hence by \cite[056Q]{stacks-project}, is again a relative effective Cartier divisor. The relative Picard functor $P_{\mathcal{X}/S}$ is represented by the base change of $\Pic_{\underline{\mathcal{X}}/\Spec(\mathbb{Z})}$ along the natural morphism from $S$ to $\Spec(\mathbb{Z})$, which is simply the constant group scheme $\underline{\mathbb{Z}^{\oplus l}}_S$.
\end{proof}

For $s\in S$, by \Cref{groupembedtocompac}, the geometric fiber $\mathcal{X}_{\overline{k(s)}}$ is the wonderful compactification of $G_{\overline{k(s)}}$, and $\overline{\Omega}_{\overline{k(s)}}$ is the big cell of $\mathcal{X}_{\overline{k(s)}}$. We denote as $h_i$ the element of $\overline{T}_{\overline{k(s)}}$ which has $i$-th coordinate $0$, and $1$ otherwise, here $i\in\{1,...,l\}$. By \cite[Theorem~6.1.8 (ii)]{BrionKumar}, the complement $\mathcal{X}_{\overline{k(s)}}\backslash G_{\overline{k(s)}}$ is the union of prime divisors $\mathbf{X}_1,...,\mathbf{X}_l$ with normal crossings, where $\mathbf{X}_i$ is the $(G_{\overline{k(s)}}\times_{\overline{k(s)}}G_{\overline{k(s)}})$-orbit of the point $h_i$. A relative version of this result is the following.

\begin{proposition}\label{boundarydivisors}
    The complement $\mathcal{X}\backslash G$ is covered by $(G\times G)$-invariant smooth relative effective Cartier divisors $S_{\alpha_i}$ with $S$-relative normal crossings\footnote{Following the terminology of \cite[exposé~XIII, 2.1]{SGA1}, for a scheme $X$ over a scheme $S$, we say that a relative Cartier divisor $D\subset X$ is strictly with $S$-relative normal crossings, if there exists a finite family $(f_i\in \Gamma(X,\mathcal{O}_X))_{i\in I}$ such that 
    \begin{itemize}
        \item $D=\bigcup_{i\in I}V_X(f_i)$;
        \item for every $x \in \Supp(D)$, $X$ is smooth at $x$ over $S$, and 
     the closed subscheme $V((f_i)_{i\in I(x)})\subset X$ is smooth over $S$ of codimension $ \vert I(x)\vert$, where $I(x)=\{i\in I\vert f_i(x)=0\}$. 
    \end{itemize}
    The divisor $D$ is with $S$-relative normal crossings, if étale locally on $X$ it is strictly with $S$-relative normal crossings. 
    }, 
    for $\alpha_i \in \Delta$, such that after passing to each geometric fiber, $S_{\alpha_i}$ becomes the boundary divisor $\mathbf{X}_i$. 
\end{proposition}

\begin{proof}
    We first prove the proposition when $S=\Spec(\mathbb{Z})$. By the same argument as in the proof of \Cref{boundarydivisor}, $\mathcal{X}\backslash G$ is of pure codimension 1. We view $\mathcal{X}\backslash G$ as a reduced closed subscheme of $\mathcal{X}$. Consider an irreducible component $\mathbf{C}$ of $\mathcal{X}\backslash G$. As $\mathcal{X}\backslash G$ is $(G\times G)$-stable and $G$ is connected, $\mathbf{C}$ is $(G\times G)$-stable. Since $\mathbf{C}=\bigcup_{g\in G\times G} g\cdot(\mathbf{C}\bigcap \overline{\Omega})$ (this follows from $\mathcal{X}=\bigcup_{g\in G\times G} g\cdot\overline{\Omega}$, cf. \Cref{defintionofcompactification}) and $\mathbf{C}$ is of codimension 1 in $\mathcal{X}$, $\mathbf{C}\bigcap \overline{\Omega}$ is of codimension 1 in the big cell $\overline{\Omega}$. Note that now the coordinate ring of $\overline{\Omega}$ is a polynomial ring over $\mathbb{Z}$ which is a UFD, hence by the Krull's Hauptidealsatz (see, for instance, \cite[Chapter~I, Theorem~1.11A]{HartshorneGTM52}), $\mathbf{C}\bigcap \overline{\Omega}$ is a hypersurface in $\overline{\Omega}$, i.e., defined by a single element of the coordinate ring of $\overline{\Omega}$. Notice that $\mathbf{C}\bigcap \overline{\Omega}$ is irreducible and $(U^-T\times U^+)$-stable, hence it is nothing else but the hypersurface defined by the coordinate $\mathbb{X}_{\alpha_i}$ of the big cell $\overline{\Omega}$ for a simple root $\alpha_i\in\Delta$. Since $\mathcal{X}$ is covered by the $(G\times G)$-translates of $\overline{\Omega}$ whose local rings now are all UFDs, by \cite[Chapter~II, Proposition~6.11]{HartshorneGTM52}, the prime Weil divisor $\mathbf{C}$ is locally principal. The assertion of the proposition about geometric fibers follows from the construction of the embedding $\mathbf{i}$ as in the proof of \Cref{groupembedtocompac}. Finally, since $\mathbf{C}$ is flat over $\Spec(\mathbb{Z})$ (\cite[Chapter~III, Proposition~9.7]{HartshorneGTM52}), by \cite[Lemma~062Y]{stacks-project}, the locally principal subscheme $\mathbf{C}$ in $\mathcal{X}$ is a relative effective Cartier divisor because so are all of its geometric fibers. The divisor $\mathcal{X}\backslash G$ is of $S$-relative normal crossings because $(\mathcal{X}_{S'}\backslash  G_{S'})\bigcap(g_1,g_2)\overline{\Omega}_{S'}$ is of strict $S$-relative normal crossings in $(g_1,g_2)\overline{\Omega}_{S'}$ for any étale $S$-scheme $S'$ and any $(g_1,g_2)\in (G\times G)(S')$.

    Back to the original base $S$, by the functorial description of the boundary divisors of $\underline{\mathcal{X}}$ in the previous case, we deduce that the boundary $\mathcal{X}\backslash G$ is covered by the base changes of the relative Cartier divisors in the boundary $\underline{\mathcal{X}}\backslash \underline{G}$ along the morphism $S\longrightarrow \Spec(\mathbb{Z})$, which are still relative effective Cartier divisors by \cite[Lemma~056Q]{stacks-project}. The $S$-relative normal crossings property of $\mathcal{X}\backslash G$ is inherited from that of $\underline{\mathcal{X}}\backslash \underline{G}$.
\end{proof}

\begin{remark}
    When the base $S$ is the spectrum of an algebraically closed field, $\mathcal{X}$ coincides with the classical wonderful compactification of $G$. In this case, any partial intersection of the boundary divisors equals to the closure of a $(G\times G)$-orbit as we recall in \Cref{introclassical}. A natural question is: does this property hold over more general base? We first clarify the notion of an orbit for a group scheme action which is
used in this remark. For an $S$-scheme $Y$ acted on by a group scheme $H$ over $S$ and an $S$-valued
point $x\in Y(S)$, we refer the subsheaf of $Y$ spanned by $x$ under the action of $H$ as the orbit of
$x$ (compare with \cite[Definition~0.4]{GIT}). Then the question to the above question in general is no. For instance, if $S$ is the spectrum of a DVR $R$ and $G=\PGL_2$, the $(G\times G)$-orbits in $\mathcal{X}$ are parametrized by the valuations of $\overline{T}(R)=\mathbb{A}_1(R)=R$ while in this case there is only one boundary divisor which is $\mathcal{X}\backslash G$. Also, from this example, we see that the set of the $(G\times G)$-orbits of $\mathcal{X}$ depends essentially on the geometry of the base scheme $S$.
\end{remark}

\section{Descent}\label{sectiondescent}
An adjoint reductive group is descended from a split adjoint reductive group. In this section, we deduce an equivariant compactification of a general adjoint reductive group scheme by twisting the equivariant compactification (defined in \Cref{sectioncompactification}) of a split form of this group. Note that a similar twisting procedure for classical wonderful compactifications appears in \cite[5.5]{twistingwonderfulcompacctification}.

\subsection{Descent datum of groups}
In this part, we review generalities of the desecent datuam for a reductive group scheme which will be important for us to construct descent datum of our compactification in split case.

\subsubsection{The setup}
We consider an adjoint reductive group scheme (not necessarily split) $G$ over a scheme $S$. By \cite[exposé~XXII, corollaire~2.3]{SGA3III} and working Zariski locally on the base $S$, without loss of generality, we assume that $G$ splits after base change to an étale cover $U$ of $S$, and, by \cite[exposé~XXV, théorème~1.1]{SGA3III}, there exists a split adjoint group $G_0$ over $S$ such that $G$ is a form of $G_0$. We fix a maximal split torus $T$ of $G_0$ as in \cite[exposé~XXII, définition~1.13]{SGA3III} with $\rk(T)=n$, in a Borel subgroup $B$ of $G_0$, and denote by 
$$\mathbf{R}\coloneq (X,\Psi, \Delta, X^{\vee}, \Psi^{\vee}, \Delta^{\vee})$$
the corresponding based root datum, where $X$ and $X^{\vee}$ are the character lattice and the cocharacter lattice of $T$. We also fix a pinning $\{X_{\alpha}\}_{\alpha\in\Delta}$ of $G_0$, where $X_{\alpha}\in \Gamma(S, \mathfrak{g}_\alpha)^{\times}$ and $\mathfrak{g}_\alpha$ is the weight subsheaf of the Lie algebra $\mathfrak{g}$ of $G_0$ with respect to the weight $\alpha$. By \Cref{groupembedtocompac} and \Cref{projectivity}, we obtain a projective smooth $(G_0\times_S G_0)$-scheme $\mathcal{X}_0$ that equivariantly contains $G_0$ as an open $S$-dense subscheme.

\subsubsection{Descent datum via \v{C}ech 1-cocycle}
Denote by $U'$ (resp., $U''$) the 2-fold (resp., 3-fold) product of $U$ over $S$. By \cite[Example~7.1.4]{redctiveconrad}, the group $G$ is descended from $G_0\times_S U$ via a \v{C}ech 1-cocycle $$\xi\in \text{Z}^1(U/ S,\Aut_S(G_0)).$$ More precisely, this $\xi$ is an automorphism of the group scheme $\widetilde{G}:=G_0\times_S U'$ over $U'$ such that $\pr_{13}^*(\xi)=\pr_{23}^*(\xi)\pr_{12}^*(\xi)$, where $\pr_{ij}:G_0\times_S U''\longrightarrow G_0\times_S U'$ be the base change along $G_0\longrightarrow S$ of the projection from $U''$ onto the product of factors $U$ with indices $i,j$.

\subsubsection{Decomposition of group automorphism}
By \cite[exposé~XXIV, théorème~1.3]{SGA3III}, we have the following short exact sequence of group schemes:
\begin{equation}\label{autogrpexactseq}
    1\longrightarrow \widetilde{G}\longrightarrow\Aut_{U'}(\widetilde{G})\longrightarrow\Out_{U'}(\widetilde{G})\longrightarrow 1.
\end{equation}
Further, the induced pinning $\{\widetilde{X}_a\}_{a\in\Delta}$ on $\widetilde{G}$ from the pinning $\{X_a\}_{a\in\Delta}$ splits the above short exact sequence and yields an isomorphism:
\begin{equation}\label{simdirectprodautogrp}
    \Aut_{U'}(\widetilde{G})\simeq \widetilde{G}\rtimes\Out_{U'}(\widetilde{G}).
\end{equation}
At the same time, since $\widetilde{G}$ is split, by \emph{loc.~cit.}, the outer automorphism group scheme $\Out_{U'}(\widetilde{G})$ is identified with the constant group scheme $\underline{E}_{U'}$ over $U'$, where $E$ is the group of automorphisms of the based root datum $(X,\Psi, \Delta, X^{\vee}, \Psi^{\vee}, \Delta^{\vee})$.

\subsubsection{Concrete Description of outer automorphism}
By \cite[exposé~XXIV, théorème~1.3 (iii)]{SGA3III}, every element $\eta\in \Out_{U'}(\widetilde{G})$ is induced by an automorphism 
 $$(\rho, ^t\rho)\in\Aut(\mathbf{R})$$
of the based root datum $\mathbf{R}$, where $\rho$ (resp., $^t\rho$) is an automorphism of the abelian group $X$ (resp., $X^{\vee}$) which induces a bijection on $\Psi$ (resp., $\Psi^{\vee}$) and preserves simple roots (resp., simple coroots) (see \cite[exposé~XXI, définition~6.1]{SGA3III}). Now $\eta$ is determined in the following way (see \cite[exposé ~XXIII, proof of théorème~4.1 and exposé~XXIII, 1.5]{SGA3III}):
\begin{itemize}[leftmargin=*]
    \item $\rho$ induces an automorphism $\rho_T$ of $T$:
$$\rho_T=\text{D}_{U'}(\rho)\colon T=\text{D}_{U'}(X)\longrightarrow \text{D}_{U'}(X)=T;$$
    \item For each simple root $\alpha\in\Delta$, set
    $$n_{\alpha}\coloneq p_{\alpha}(1)p_{-\alpha}(-1)p_{\alpha}(1),$$
    where $p_{\alpha}: \mathbb{G}_{a,U'}\rightarrow U_{\alpha}$ is given by the corresponding $X_{\alpha}$ in the pinning we fixed. Then $\rho_T$ extends to an automorphism $\rho_N$ of $\Norm_G(T)$ by taking
    $n_{\alpha}$ to $n_{\rho(\alpha)}$.
    \item For each simple root $\alpha\in\Delta$, set 
    \begin{flalign*}
    \rho_{\alpha}\colon U_{\alpha}&\longrightarrow U_{\rho(\alpha)},\\
    p_{\alpha}(x)&\longmapsto p_{\rho(\alpha)}(x).
    \end{flalign*}
    \item For each root $r\in\Psi$, there exists an $n\in\Norm_{\widetilde{G}}(T)(U')$ such that $\overline{n}(r)\in\Delta$, where $\overline{n}$ is the image of $n$ in the Weyl group $W$. Thus, by \cite[exposé~XXIII, lemme~2.3.5]{SGA3III}, we have a homomorphism 
    \begin{flalign*}
        \rho_r\colon U_r &\longrightarrow U_{\rho(r)}, 
    \end{flalign*}
    such that 
       $ \Ad_{\rho_N(n)}\rho_{r}(x)=\rho_{\overline{n}(r)}(\Ad_{n}x)$,
    for any $x\in U_r(V)$ and $V\rightarrow U'$.
\end{itemize}
Then, as in \cite[exposé~XXIII, proof of théorème~4.1]{SGA3III}, the above data glue together uniquely to yield the group automorphism $\eta$ of $\widetilde{G}$ over $U'$. More precisely, as in \emph{loc. cit.}, we can prove the existence of the group automorphism $\eta$ by induction on the semisimple rank of $\widetilde{G}$. If the semisimple rank of $\widetilde{G}$ is $0$, then we are in the toral case which is clear, see \cite[exposé~XXIII, 4.1.1]{SGA3III}. The $\SL_2$-crutch is used to deal with the case when the semisimple rank of $\widetilde{G}$ is $1$, see \cite[exposé~XXIII, 4.1.2]{SGA3III}. When the semisimple rank of $\widetilde{G}$ is $2$, the $\eta$ is constructed via a case-by-case argument \cite[exposé~XXIII, 4.1.4-4.1.8]{SGA3III} based on \cite[exposé~XXIII, corollaire~2.5]{SGA3III}. When the semisimple rank of $\widetilde{G}$ is strictly larger than $2$, we are reduced to the semisimple rank $2$ case by the existence criterion for group homomorphisms \cite[exposé~XXIII, corollaire~2.4]{SGA3III}, see \cite[exposé~XXIII, 4.1.9]{SGA3III} for this reduction step.

\subsection{Descent datum for compactifications}
Now we are in the position to define a descent datum on $\mathcal{X}_{0}\times_S U'$. Since the construction in \Cref{defintionofcompactification} commutes with any base change, then $\mathcal{X}_0 \times_S U'$ is identified to the compactification of $\widetilde{G}$, denoted by $\widetilde{\mathcal{X}}$.
By \Cref{simdirectprodautogrp}, we can uniquely decompose $\xi$ into the product of $\delta\in\widetilde{G}(U')$ and $\eta\in\Out_{U'}(\widetilde{G})(U')$.

 We are going to define two automorphisms of  $\widetilde{\mathcal{X}}$ which (necessarily uniquely) extend $\delta$ and $\eta$ respectively and then, take their composition as an extension of $\xi$ on $\widetilde{\mathcal{X}}$. 

For the element $\delta$, we define its action on $\widetilde{\mathcal{X}}$ to be the conjugation action, see \Cref{definitionofgrpaction}.

For the element $\eta$, inspired by the above explicit description of $\eta$, we define an automorphism $\rho_{\overline{T}}$ of $\overline{T}=\prod_{\Delta}(\mathbb{G}_{a,U'})$ to be the permutation of coordinates induced by $\rho$ so that the following diagram    
$$\xymatrix{
T \ar@{^{(}->}[r] \ar[d]^{\rho_T}& \overline{T}\ar[d]^{\rho_{\overline{T}}}\\
T \ar@{^{(}->}[r] & \overline{T}
}$$
commutes, where the two horizontal morphisms are the embedding of $T$ induced by negative simple roots as \Cref{Omegaembedformula}. As $\rho$ preserves $\Delta$, the $\eta$ defines an automorphism $\hat{H}$ of $\widetilde{G}\times_{U'}\overline{\Omega}_{U'}\times_{U'} \widetilde{G}$ by applying $\rho_{\overline{T}}$ and $\rho_r$, $r\in\Psi$ to each component, where $\overline{\Omega}_{U'}$ is the base change $\overline{\Omega}\times_{S}U'$.

\begin{lemma}\label{outmorphismpreserverelation}
    The automorphism $\hat{H}$ of   $\widetilde{G} \times_{U'}\overline{\Omega}_{U'}\times_{U'} \widetilde{G}$ preserves the relation defined in \Cref{relationonOmega}.
\end{lemma}

\begin{proof}
    In the view of \Cref{relationonOmega}, it suffices to show that the following diagram of $U'$-rational morphisms
    $$\xymatrix{
\widetilde{G}\times_{U'}\overline{\Omega}_{U'}\times_{U'} \widetilde{G} \ar[r]^{\hat{H}}_{\cong}\ar@{-->}[d]^{\pi_{U'}}  &  \widetilde{G}\times_{U'}\overline{\Omega}_{U'}\times_{U'} \widetilde{G}\ar@{-->}[d]^{\pi_{U'}}\\
\overline{\Omega}_{U'}\ar[r]_{\cong}^{\hat{H}\vert_{\overline{\Omega}_{U'}}} &\overline{\Omega}_{U'}
    }$$
    commutes, where $\pi_{U'}$ is the base change of the $S$-rational morphism $\pi$ defined in \Cref{theoremrationalaction} and the dotted vertical arrows indicate that $\pi_{U'}$ is a $U'$-rational morphism.
    Since $\hat{H}$ is induced by $\eta$ which is an automorphism of $\widetilde{G}$, the commutativity of the diagram then follows from \Cref{theoremrationalaction}.
\end{proof}

Thanks to \Cref{outmorphismpreserverelation} , combined with \Cref{defintionofcompactification}, $\hat{H}$ induces an automorphism $H$ of $\widetilde{\mathcal{X}}$. We take $H$ as the  action induced by $\eta$ on $\widetilde{\mathcal{X}}$. We thus obtain an automorphism $\varphi$ of $\widetilde{\mathcal{X}}$ by composing $H$ and the conjugation action of $\delta$. We claim that $\varphi$ is a descent datum for the projective scheme $\mathcal{X}_0\times_S U'$. In fact, it suffices to check the cocycle condition on $\varphi$, which follows from the cocycle condition on $\xi$ and the schematic density of $G_0$ in $\mathcal{X}_0$. 

\subsection{Effectivity of descent datum}
A descent datum for a projective scheme, in general, is not effective (see \cite[08KE]{stacks-project}). In order to make our descent datum effective, we are going to construct an ample line bundle compatible with the descent datum $\varphi$.



\begin{theorem}\label{descentampleCartier}
    There exists an ample effective Cartier divisor $\mathbf{D}$ on $\widetilde{\mathcal{X}}$ such that $\mathbf{D}$ is stable under the automorphism $\varphi$ of $\widetilde{\mathcal{X}}$, and $\mathbf{D}$ descends to an ample effective Cartier $\mathbf{D}_0$ of $\mathcal{X}_0$.
\end{theorem}

\begin{proof}
    Let $\gamma\coloneq \Sigma_{i=1}^n n_i \alpha_i$ be the sum of the positive roots with respect to the based root datum $\mathbf{R}$. Then $\gamma$ is invariant under $\rho$. By \cite[Part~II, Chapter~1, 1.5(6)]{Repjantzen}, $\gamma$ is regular dominant. Consider the effective Cartier divisor 
    $$\mathbf{D}:=\Sigma_{i=1}^n n_i S_{\alpha_i},$$ 
    where $S_{\alpha_i}$ are the boundary effective Cartier divisor of $\widetilde{\mathcal{X}}$ from \Cref{boundarydivisors}. By \cite[corollaire 9.6.5]{EGAIV2}, being relative ample can be checked fiberwise. Hence, to show that $\mathbf{D}$ is $U'$-ample, we are left to check that, for each point $u\in U'$, the geometric fiber $\mathbf{D}_{\overline{k(u)}}=\Sigma_{i=1}^n n_i (S_{\alpha_i})_{\overline{k(u)}}$ is an ample divisor of the geometric fiber $\widetilde{\mathcal{X}}_{\overline{k(u)}}$ which, by \Cref{groupembedtocompac}, is identified with the wonderful compactification of $\widetilde{G}_{\overline{k(u)}}$. By \Cref{boundarydivisors}, $(S_{\alpha_i})_{\overline{k(u)}}$ is a boundary divisor in  $\widetilde{\mathcal{X}}_{\overline{k(u)}}\backslash\widetilde{G}_{\overline{k(u)}}$. By \cite[Proposition~6.1.11 (ii)]{BrionKumar}, $\mathcal{L}_{\widetilde{G}_{\overline{k(u)}}}((S_{\alpha_i})_{\overline{k(u)}})$ is the line bundle $\mathcal{L}_{\widetilde{G}_{\overline{k(u)}}}(\alpha_i)$ which is defined in \cite[Proposition~6.1.11 (i)]{BrionKumar}. Thus we have that 
    \begin{flalign*}
 \mathcal{L}_{\widetilde{G}_{\overline{k(u)}}}(\mathbf{D}_{\overline{k(u)}})=\sum_{i=1}^n n_i \mathcal{L}_{\widetilde{G}_{\overline{k(u)}}}((S_{\alpha_i})_{\overline{k(u)}})=\sum_{i=1}^n n_i \mathcal{L}_{\widetilde{G}_{\overline{k(u)}}}(\alpha_i)=\mathcal{L}_{\widetilde{G}_{\overline{k(u)}}}(\sum_{i=1}^n n_i \alpha_i)=\mathcal{L}_{\widetilde{G}_{\overline{k(u)}}}(\gamma),
    \end{flalign*}
where the third equality follows from \emph{loc. cit}.
 As $\gamma$ is regular dominant, by \cite[Proposition~6.1.11(iii)]{BrionKumar}, $\mathcal{L}_{\widetilde{G}_{\overline{k(u)}}}(\gamma)$ is ample. Hence $\mathbf{D}$ is an ample effective Cartier divisor.
    
Now we show that $\mathbf{D}$ is stable under $\varphi$. Recall that $\varphi$ is the composition of the conjugate action of an element of $\widetilde{G}(U')$ and the automorphism $H$ defined by the automorphism $(\rho,^t \rho)$ of the based root datum $\mathbf{R}$. By \Cref{boundarydivisors}, each $S_{\alpha_i}$ is $(\widetilde{G}\times_{U'}\widetilde{G})$-stable, hence so is $\mathbf{D}$. Moreover, by \emph{loc. cit.,} we see that $H$ takes $S_{\alpha_i}$ to $S_{\alpha_{\rho(i)}}$, and their coefficients $n_i$ and $n_{\rho(i)}$ in $\mathbf{D}$ are equal because of the invariance of $\gamma$ under $\rho$, thus $\mathbf{D}$ is stable under $H$ as well, as desired.

By applying \Cref{boundarydivisors} to $\mathcal{X}_0$, we get divisors $S_{0\,\alpha_i}$ on the boundary of $G_0$ in $\mathcal{X}_0$.
Notice that, by the fibral description of $S_{\alpha_i}$ in \Cref{boundarydivisors}, each $S_{\alpha_i}$ descends to $S_{0\,\alpha_i}$. Thus $\mathbf{D}$ descends to $\mathbf{D}_0=\Sigma_{i=1}^n n_iS_{0\,\alpha_i}$ which is an ample effective divisor by the same argument as above.
\end{proof}

\begin{corollary}\label{solutiontoconjecadjoint}
    The $(G_0\times_S U)\times_U(G_0\times_S U)$-scheme $\mathcal{X}_0 \times_S U$ over $U$ descends to a smooth projective $(G\times_S G)$-scheme $\mathcal{X}$ over $S$ such that  the geometric fibers of $\mathcal{X}$ are the wonderful compactifications of the corresponding geometric fibers of $G$, and $\mathcal{X}$ contains $G$ as an $S$-dense open subscheme. Moreover, the boundary $\mathcal{X}\backslash G$ is a smooth relative Cartier divisor with $S$-relative normal crossings.
\end{corollary} 

\begin{proof}
    By \cite[Section~6.1, Theorem~7]{BLR}, the descent datum $\varphi$ of the pair $$(\mathcal{X}_0\times_S U, \mathcal{L}_{\mathcal{X}_0 \times_S U}(\mathbf{D}_0\times_S U))$$
    is effective, where $\mathbf{D}_0$ is given in \Cref{descentampleCartier}. Thus $\mathcal{X}_0\times_S U$ descends to a scheme $\mathcal{X}$ over $S$, which is quasi-projective over $S$ because $\mathcal{L}_{\mathcal{X}_0 \times_S U}(\mathbf{D}_0\times_S U)$ descends to a relative ample line bundle on $\mathcal{X}$. By \cite[Proposition~(2.7.1) (vii)]{EGAIV2}, this $\mathcal{X}$ inherits properness from $\mathcal{X}_0$, hence $\mathcal{X}$ is projective over $S$.

    In order to descend the action of $(G_0\times_S U)\times_U(G_0\times_S U)$ on the scheme $\mathcal{X}_0 \times_S U$, we have to show that the following diagram commutes:
    $$\xymatrix{
\widetilde{G}\times_{U'}\widetilde{\mathcal{X}}\times_{U'}\widetilde{G}  \ar[r]^-{\widetilde{\Theta}} \ar[d]^-{\xi\times\varphi\times\xi}  & \widetilde{\mathcal{X}} \ar[d]^-{\varphi}\\
\widetilde{G}\times_{U'}\widetilde{\mathcal{X}}\times_{U'}\widetilde{G}  \ar[r]^-{\widetilde{\Theta}}  & \widetilde{\mathcal{X}},
}
$$
where $\widetilde{\Theta}$ is the action of $\widetilde{G}\times_{U'}\widetilde{G}$ on $\widetilde{\mathcal{X}}$ defined in \Cref{definitionofgrpaction}. Recall that $\xi$ (resp., $\varphi$) consists of an inner part and an outer part, hence it suffices to verify the above diagram separately. For the inner part, the commutativity follows from the group action law on $\widetilde{X}$; for the outer part, it follows from the definition of $\varphi$. Then by \cite[Section~6.1, Theorem~6 (a)]{BLR}, this $\widetilde{\Theta}$ descends to a morphism $\Theta\colon G\times_S\mathcal{X}\times_S G\longrightarrow \mathcal{X}$ which is again a group action.

Since $\varphi$ restricts to $\xi$, again, by \cite[Section~6.1, Theorem~6~(a)]{BLR}, the open embedding $G_0\times_S U\hookrightarrow \mathcal{X}_0\times_S U$ defined in \Cref{groupembedtocompac} descends to a morphism $i\colon G \longrightarrow \mathcal{X}$ which is an open immersion by \cite[Proposition~2.7.1(x)]{EGAIV2}. Finally, the fiberwise claim on $\mathcal{X}$ follows from the corresponding results on $\mathcal{X}_0$, see \Cref{groupembedtocompac}. 

By \Cref{boundarydivisors}, the complement $\mathcal{X}_0\backslash G_0$ is of $S$-relative normal crossings. Hence, since $\mathcal{X}\backslash G$ is descended from $(\mathcal{X}_0\backslash G_0)\times_S U$ by the definition of being of relative normal crossings \cite[exposé~XIII, 2.1]{SGA1}, $\mathcal{X}\backslash G$ is of $S$-relative normal crossings as well.
\end{proof}

\begin{remark}
    In \cite{Wedhornspherical}, Wedhorn introduced the notion of a spherical space over an arbitrary base scheme as a generalization of a spherical variety over an algebraically closed field. More precisely, in loc. cit., a spherical space for a reductive group scheme $G$ over a scheme $S$ is defined to be a flat, separated, and finitely presented algebraic space $X$ over $S$ together with an action by $G$ such that for any $s\in S$, $X_{\overline{k(s)}}$ is a spherical $G_{\overline{k(s)}}$-variety in the sense that $X_{\overline{k(s)}}$ has an open schematically dense orbit under a (hence all) Borel subgroup of $G_{\overline{k(s)}}$. Since a classical wonderful compactification is spherical, \Cref{solutiontoconjecadjoint} produces plenty of meaningful examples of spherical spaces. 
\end{remark}

\subsection{Picard group scheme in general case}\label{picardgroupnonsplit}
By Grothendieck's result on representability of relative Picard functor \cite[n°~232, théorème~3.1]{FGA}, the relative Picard functor $P_{\mathcal{X}/S}$ (recalled in \Cref{boundarydivisor}) is represented by a scheme. In this section, we will study the Picard group scheme of $\mathcal{X}$ for general adjoint reductive group scheme which is not assumed to be split. The base is assumed to be a field $k$ in this section. 

\subsubsection{$\ast$-action} In this part, we recall the definition of the $\ast$-action defined by Borel and Tits in \cite[\S~6.2]{BorelTits}.
Let $k_s$ be the separable closure of $k$. Then $G_{k_s}$ splits over $k_s$. Let $\mathfrak{g}$ be the Lie algebra of $G_{k_s}$. Let $l$ be the rank of $G_{k_s}$. Let $T\subset G$ be a maximal $k$-torus which exists by Grothendieck's \cite[exposé~XIV, théorème~1.1]{SGA3II}. Let $\Phi\coloneq \Phi(G_{k_s},T_{k_s})$ be the set of roots with respect to $T_{k_s}$. We fix a pinning of $G_{k_s}$ 
$$\mathcal{E}\coloneq \{T_{k_s}\subset B \subset G_{k_s}, \Delta\subset X^*(T_{k_s}), \{g_{\alpha}\in \Gamma (S,\mathfrak{g}_{\alpha})\backslash 0\vert\alpha\in\Delta\}\},$$
where $\mathfrak{g}= \bigoplus_{\alpha\in\Phi}\mathfrak{g}_{\alpha}$ is the weight decomposition, $B$ is a Borel subgroup of $G_{k_s}$ and $\Delta$ is the set of simple roots defined by $B$. Let $\Delta=\{\alpha_1, \alpha_2,...,\alpha_l\}$. Let $\Phi^+$ be the corresponding positive system of roots for $\Phi$. Then, for an element $\gamma\in \Gal(k_s/k)$, $\gamma\cdot \Phi^+$ is again a positive system of roots for $\Phi$. By \cite[exposé~XXIV, lemme~1.5]{SGA3III}, there is a unique inner automorphism $w_{\gamma}$ of $G_{k_s}$ which sends $\gamma\cdot \Phi^+$ to $\Phi^+$. We define the $\ast$-action of $\Gal(k_s/k)$ on $\Delta$ by 
$$\gamma\ast r\coloneq w_{\gamma}\cdot \gamma\cdot r,\; r\in\Delta.$$
Note that if $G$ is quasi-split over $k$, then the $\ast$-action is simply the natural action of $\Gal(k_s/k)$ on $\Delta$. An intrinsic way to express $\ast$-action is to introduce the notion of a scheme of Dynkin diagrams, see \cite[Remark~7.1.2]{redctiveconrad} and \cite[exposé~XXIV, \S~3]{SGA3III}. The $\ast$-action clearly respects the structure of Dynkin diagrams.

\subsubsection{Descent of Picard group scheme}
The natural action of the Galois group $\Gal(k_s/k)$ on $G_{k_s}$ induces an action of $\Gal(k_s/k)$ on the Picard group $\Pic(\mathcal{X}_{k_s})$. An element $\gamma\in \Gal(k_s/k)$ sends the divisor $\overline{B\dot{s_i}B^-}$ to the divisor $\overline{\gamma(B)\gamma(\dot{s_i})\gamma(B^-)}$. Note that $\gamma$ transfers the pinning $\mathcal{E}$ into a pinning 
$$\gamma(\mathcal{E})\coloneq \{T_{k_s}\subset \gamma(B) \subset G_{k_s}, \gamma(\Delta)\subset X^*(T_{k_s}), \{\gamma(g_{\alpha})\in \Gamma (S,\mathfrak{g}_{\gamma\cdot\alpha})\backslash 0\vert \alpha\in\Delta \}\},$$
where $\gamma(\Delta)$ is defined by the natural \emph{left} action of $\Gal(k_s/k)$ on $X^*(T_{k_s})$. Moreover $\gamma(\dot{s_i})$ is a representative of the simple reflection along $\gamma(\alpha_i)$.

By \cite[exposé~XXIV, lemme~1.5]{SGA3III}, there is a unique inner automorphism of $G_{k_s}$ which sends $\gamma(\mathcal{E})$ to $\mathcal{E}$; since $G$ is adjoint over $k$, the inner automorphism is given by a unique element $g_{\gamma}\in G(k_s)$. More precisely, we have that $\Ad_{g_{\gamma}}(T_{k_s})=T_{k_s}$ and that the conjugation $\Ad_{g_{\gamma}}$ naturally induces a \emph{left} action on $X^*(T_{k_s})$ which sends $\gamma(\Delta)$ to $\Delta$. Now since $G_{k_s}\times_{k_s}G_{k_s}$ acts trivially on the discrete group $\Pic(\mathcal{X}_{k_s})$, we have an equivalent of divisors $$\overline{\gamma(B)\gamma(\dot{s_i})\gamma(B^-)}\sim\Ad_{g_{\gamma}}(\overline{\gamma(B)\gamma(\dot{s_i})\gamma(B^-)})=\overline{B\Ad_{g_{\gamma}}(\gamma(\dot{s_i}))B^-},$$
where $\Ad_{g_{\gamma}}(\gamma(\dot{s_i}))$ is a representative of the simple reflection along the image of $\alpha_i$ under the $\ast$-action of $\gamma$.

Recall that, by \Cref{boundarydivisor}, the pinning $\mathcal{E}$ of $G_{k_s}$ 
gives an isomorphism between the relative Picard group scheme $\underline{\Pic_{k_s}}(\mathcal{X}_{k_s})$ and the constant group scheme $\underline{\mathbb{Z}^{\oplus \Delta}}_{k_s}$ by sending $\overline{B\dot{s_i}B^-}$ to the element of $\mathbb{Z}^{\oplus \Delta}$ indexed by $\alpha_i$.
The above argument shows that, under this isomorphism, the Galois action of $\Gal(k_s/k)$ on $\underline{\Pic_{k_s}}(\mathcal{X}_{k_s})$ amounts to the action of $\Gal(k_s/k)$ on the constant group scheme $\underline{\mathbb{Z}^{\oplus \Delta}}_{k_s}$ induced by the $\ast$-action of $\Gal(k_s/k)$ on $\Delta$. Let $\mathbb{O}(\mathbb{Z}^{\oplus \Delta})$ be the set of the $\Gal(k_s/k)$-orbits in $\mathbb{Z}^{\oplus \Delta}$. Then, by Galois descent and choosing a representative $\dot{o}$ for each orbit $o\in \mathbb{O}(\mathbb{Z}^{\oplus \Delta})$, we have an (non canonical) isomorphism of $k$-schemes
$$\underline{\Pic_k}(\mathcal{X})\xlongrightarrow{\cong}\coprod_{o\in \mathbb{O}(\mathbb{Z}^{\oplus \Delta})}\Spec(k_o),$$
where $\underline{\Pic_k}(\mathcal{X})$ is the Picard group scheme of $\mathcal{X}$ and $k_o\subset k_s$ is the subfield fixed by the stabilizer of $\dot{o}$ in $\Gal(k_s/k)$.

\subsubsection{Examples}
Now let us see some examples of the Picard groups of our compactifications of forms of absolute simple groups. 

If $G$ is an \emph{inner} form of $G_{k_s}$, then the $\ast$-action of $\Gal(k_s/k)$ on $\Delta$ is trivial. Hence the Picard scheme $\underline{\Pic_k}(\mathcal{X})$ is isomorphic to the constant scheme $\underline{\mathbb{Z}^{\oplus \Delta}}_{k}$. This is also true for any $G$ of type $B_n$, $C_n$, $E_7$, $E_8$, $F_4$ and $G_2$ because the $\ast$-action preserves the structure of Dynkin diagrams. 

Now assume that $G$ is an \emph{outer} form of $G_{k_s}$. By \cite[corollaire~7.8.8]{EGA3II} and \cite[\S~8.1, Proposition~4]{BLR}, we have $\Pic(\mathcal{X})=\underline{\Pic_k}(\mathcal{X})(k)$. If $G$ is of type $A_n, n\geq 2$ (resp., $D_n, n\textgreater 4$; resp., $D_4$; resp., $E_6$), the Picard group $\Pic(\mathcal{X})$ is isomorphic to the free abelian group of rank $\lceil n/2 \rceil$ (resp., $n-1$; resp., $2$ or $3$; resp., $4$), where $\lceil n/2 \rceil$ is the smallest integer greater than or equal to $n/2$. All aforementioned cases can happen according to the classification of all possible Tits indexes of absolutely simple groups \cite[Table~II]{Titsclassificationsemisimplegroups}.

\section{Applications}\label{sectionapplications}
In this section, we discuss applications of the compactification constructed in \Cref{sectiondescent} to torsors of reductive group schemes. 

The initial motivation of this paper comes from the following conjecture about equivariant compactifications for reductive group schemes by \v{C}esnavi\v{c}ius:

\begin{conjecture}\label{conjecture}
(\cite[Conjecture~6.2.3]{CKproblemtorsors}). For a reductive group $G$ over a Noetherian scheme $S$ such that $G$ is isotrivial in the sense that there exists a finite étale cover $S'$ of $S$ such that $G\times_{S}S'$ splits, there exists a projective, finitely presented scheme $X$ containing $G$ as an $S$-dense open subscheme, together with a left action of $G$ on $X$ that extends the left translation of $G$ on itself.
\end{conjecture}

If $S$ is a field, then such an $X$ can be obtained by simply taking the schematic closure of $G$ in some projective space which equivariantly contains $G$ as an open subscheme. The main difficulty of \Cref{conjecture} lies in the fact that, in general, the schematic image of a morphism does not commute with nonflat base change.

Note that \Cref{solutiontoconjecadjoint} gives an affirmative answer to \Cref{conjecture} for adjoint reductive group schemes, and also the left and the right $G$-translations are both extended. Notice also that the solution of \Cref{conjecture} for tori has been given, see \cite[Theorem~6.3.1]{CKproblemtorsors}. 

The following result is conditionally established in \cite[Proposition~6.2.4]{CKproblemtorsors} for isotrivial torsors of reductive group schemes under the assumption that \Cref{conjecture} holds. 

\begin{proposition}\label{compactificationtorsor}
    (i) For an isotrivial torsor $X$ under an adjoint reductive group scheme $G$ over a Noetherian scheme $S$, there exists a projective, finitely presented $G$-scheme $\overline{X}/ S$ containing $X$ as a open $S$-dense subscheme such that the $G$-action on $\overline{X}$ restricts to the $G$-action on $X$. 
    
    (ii) If $S$ is the spectrum of a semilocal Noetherian ring, then, for any closed subscheme $Z\subset S$ and any section $a\in X(Z)$, there exists a finite étale cover $\widetilde{S}$ of $S$, a morphism
    $\nu\colon Z\rightarrow\widetilde{S}$ and a section $\widetilde{a}\in X(\widetilde{S})$ whose $\nu$-pullback is $a$.
\end{proposition}

The proof of the above result needs the following (simplified) version of Bertini's theorem. For the proof, we refer to \cite[proof of Proposition~4.1.3]{CKproblemtorsors} and references therein.

\begin{proposition}\label{bertinitheorem}
    Let $X$ be a projective scheme over a field $k$, and let $Z\subset X^{\text{sm}}$ be a finte étale $k$-subscheme whose residue fields are separable over $k$. Fix an ample line bundle $\mathcal{O}_X(1)$ over $X$. For any positive integer $n\leq \dim(X)$ and a closed $k$-subscheme $Y$ whose intersection with $Z$ is empty, there exist hypersurfaces $H_1, H_2,..., H_n$ with respect to the given $\mathcal{O}_X(1)$, such that 
    \item[i)]  the intersection $X^{\text{sm}}\cap\bigcap_{1\leq i\leq n}H_i$ is smooth over $k$, is of pure dimension $\dim(X)-n$ and contains $Z$;
    \item[ii)] $\dim(Y\cap \bigcap_{i\in I}H_i)\leq \dim(Y)-\#I$ for all $I\subset\{1,2,...,n\}$.
    
    \noindent Moreover, the hypersurfaces $\{H_i\}_{1\leq i\leq n}$ can be constructed iteratively so that, for any $1\leq i\leq n$ and fixed hypersurfaces $H_1,...,H_{i-1}$, an $H_i$ can be chosen to have sufficiently large degree.
\end{proposition}

\begin{proof}[Proof of \Cref{compactificationtorsor}]
    The proposition follows from \Cref{solutiontoconjecadjoint} and \cite[Proposition 6.2.4]{CKproblemtorsors}.
    For the sake of completeness, we include a proof. The construction of $\overline{X}$ is obtained by twisting the compactification $\mathcal{X}$ of $G$ given in \Cref{solutiontoconjecadjoint}. Let $\overline{X}$ be the contracted product $X\times^{G}\mathcal{X}$, which is an algebraic space by \cite[06PH]{stacks-project}. Let $S'\rightarrow S$ be a finite étale cover which trivializes the torsor $X$. Then the base change $\overline{X}_{S'}=\mathcal{X}_{S'}$ is a quasi-projective $S'$-scheme. By \cite[Proposition~A.5.8 and its proof]{CGP}, the Weil restriction $\Res_{S'/S}(\overline{X}_{S'})$ is quasi-projective over $S$. By the étale descent for properness, $\Res_{S'/S}(\overline{X}_{S'})$ is proper. Combining with \cite[remarques~5.5.4(i)]{EGA2}, this Weil restriction is also $S$-projective. Therefore, the adjunction morphism $\overline{X}\rightarrow \Res_{S'/S}(\overline{X}_{S'})$, which is a closed immersion (étale locally by \cite[Proposition~A.5.7]{CGP}), shows that $\overline{X}$ is a projective scheme over $S$.
    
    Now we show the claim (ii). We fix a closed embedding $\overline{X}\hookrightarrow \mathbb{P}_S^n$ for some $n\in\mathbb{Z}_{\geq 0}$; by applying a linear automorphism of $\mathbb{P}_S^n$, we can assume that $a=[0:...:0:1]\in \mathbb{P}_S^n(Z)$. Let $S'\subset S$ be the union of all closed points of $S$. Let $d$ be the relative dimension of the torsor $X$. Since $X$ is $S$-dense in $\overline{X}$, we have $\dim(\overline{X}_s\backslash X_s)< d$ for all $s\in S'$. By applying \Cref{bertinitheorem} to each fiber of $\overline{X}_{S'}$ with $Z$ there being the image of $a_{S'\cap Z}$, we can find hypersurfaces $H_1, H_2,..., H_d\subset \overline{X}_{S'}$ such that each hypersurface $H_i$ for $1\leq i\leq d$ has a constant degree on $S'$-fibers of $\overline{X}_{S'}$; the intersection $\bigcap_{1\leq i\leq d}H_i\subset \overline{X}_{S'}$ contains $a_{S'\cap Z}$, lies in $X_{S'}$, and is finite étale over $S'$. By applying \cite[corollaire~2.2.4]{EGA3I}, at the cost of increasing the degrees, we can lift $\{H_i\}_{1\leq i\leq d}$ to hypersurfaces $\{\hat{H_i}\subset \overline{X}_{Z\bigcup S'}\}_{1\leq i\leq d}$ such that $\bigcap_{1\leq i\leq d}\hat{H_i}$ contains the section $a$ (this amounts to that the coefficient of the monomial, which is a power of the last coordinate of the projective space $\mathbb{P}^n_{S'\bigcup Z}$, in the defining equation of each $\hat{H_i}$ is zero). Again by applying \emph{loc. cit.,}, we can further lift $\{\hat{H_i}\}_{1\leq i\leq d}$ to hypersurfaces $\{\widetilde{H_i}\subset \overline{X}\}_{1\leq i\leq d}$. Let $\widetilde{S}\coloneq \bigcap_{1\leq i\leq d}\widetilde{H_i}$. Then $\widetilde{S}$ lies in $X$ and contains the image of $a$. The $S$-finiteness of $\widetilde{S}$ follows from our construction of $\{\hat{H_i}\}_{1\leq i\leq d}$ and \cite[01TI, 02OG]{stacks-project}. Similarly, by the openness of the étale locus \cite[02GH]{stacks-project} and \cite[02IL]{stacks-project}, $\widetilde{S}$ is étale over $S$. The natural inclusion $Z\hookrightarrow \widetilde{S}$ is the sought morphism $\nu$.
\end{proof}

\begin{remark}
    The subtlety of the final part of \Cref{compactificationtorsor} lies in arranging the finiteness of $\widetilde{S}$. Without the finiteness requirement, this final part follows from \cite[Proposition~6.1.1~(a)]{CKproblemtorsors} by henselizing along $Z$.
\end{remark}

Another application is the following trick of equating reductive group schemes, which is frequently used (in various forms) in the study of the Grothendieck--Serre conjecture (see \cite[Section~3.1]{CKproblemtorsors} for a detailed survey of this conjecture). The following result is established in \cite[Proposition 6.2.5]{CKproblemtorsors} under the assumption of \Cref{conjecture}. A pleasant point of the following proposition is that it works for reductive group schemes even if we only dispose of the wonderful compactifications for adjoint reductive group schemes. Note that \Cref{equatinggroup} has been used by \v{C}esnavi\v{c}ius and Fedorov in their proof of the Grothendieck--Serre conjecture for totally isotropic reductive groups over unramified regular local rings \cite{unramifiedisotropicGrothendieckserre}.

\begin{proposition}\label{equatinggroup}
    Let $S$ be the spectrum of a Noetherian semilocal ring whose local rings are all geoemetrically unibranch, e.g., normal (see \cite[23.2.1]{EGAIV1}), and let $Z\subset S$ be a closed subscheme. Assume that $G_1$ and $G_2$ are two reductive group schemes over $S$ that have the same root datum over each geometric fiber, and that there is an isomorphism of group schemes $\varphi: G_1\vert_{Z}\rightarrow G_2\vert_{Z}$. Then, there are a finite étale cover $\widetilde{S}$ of $S$, a morphism $c: Z\rightarrow \widetilde{S}$, and an isomorphism $\widetilde{\varphi}:G_1\times_S \widetilde{S}\rightarrow G_2\times_S \widetilde{S}$ that lifts $\varphi$ along $c$.
\end{proposition}

\begin{proof}
    The proposition follows from \Cref{solutiontoconjecadjoint} and \cite[Proposition~6.2.5]{CKproblemtorsors}. For completeness, we include a detailed proof. By \cite[exposé~XXII, corollaire~2.3]{SGA3III}, $G_1$ and $G_2$ split étale locally over $S$. Then thanks to the fiberwise assumption, $G_1$ and $G_2$ are étale locally isomorphic.  
    Thus, the quotient, say $\overline{\mathbf{I}}$, of the isomorphism scheme $\mathbf{I}:=\text{Isom}_{\text{grp}}(G_1, G_2)$ by the action of $(G_1)_{\text{ad}}$ is étale locally constant because, étale locally, $\overline{\mathbf{I}}$ is the outer automorphism scheme of a split reductive group scheme, which is, by \cite[exposé~XXIV, théorème~1.3~(iii)]{SGA3III}, a constant group scheme. 
    By descent \cite[proposition~2.6.1 (i)]{EGAIV2} and \cite[corollaire~17.7.3 (ii)]{EGAIV4}, $\overline{\mathbf{I}}$ is étale and surjective over $S$. By \cite[exposé~X, corollaire 5.14]{SGA3II} and \cite[corollaire~6.1.9]{EGA1}, $\overline{\mathbf{I}}$ decomposes as a disjoint union of connected components, which are clopen, and also finite and étale over $S$. Now consider the section $\overline{\varphi}\in\overline{\mathbf{I}}(Z)$ induced by $\varphi$, and take $\overline{Z}$ as the finite disjoint union of all connected components of $\overline{\mathbf{I}}$ that intersect the image of $Z$. Then $\overline{Z}$ is finite étale over $S$, and the natural embedding $\overline{Z}\hookrightarrow \overline{\mathbf{I}}$, denoted by $\overline{\alpha}$, gives a lifting of $\overline{\varphi}$. By base changing the $(G_1)_{\text{ad}}$-torsor $\mathbf{I}$ over $\overline{\mathbf{I}}$ along $\overline{\alpha}$ and considering the section in $(\mathbf{I}\times_{\overline{\mathbf{I}}}\overline{Z})(Z)$ which is induced by $\overline{\varphi}$ and $\overline{\alpha}$, the proposition then follows from \Cref{compactificationtorsor} (here our assumption on $S$ ensures the isotriviality condition needed in \Cref{compactificationtorsor}, see \cite[exposé~XXIV, théoreme~4.1.5, corollaire~4.1.6]{SGA3III}).
\end{proof}


\renewcommand{\bibname}{References}

\printbibliography

\end{document}